\documentclass[12pt,a4paper,twoside,reqno]{amsart}

\usepackage[all,cmtip]{xy}
\usepackage{amsmath,amscd}
\usepackage{amssymb,amsfonts}
\usepackage[driver=pdftex,margin=3cm,heightrounded=true,centering]{geometry}
\usepackage{JK}
\usepackage[colorlinks=true,linkcolor=blue,citecolor=blue]{hyperref}

\newtheorem{thm}{Theorem}[section]

\usepackage{graphicx}

\begin{document}
\title[Canonical holomorphic sections of determinant line bundles]{Canonical holomorphic sections of determinant line bundles}

\author{Jens Kaad}
\address{International School of Advanced Studies (SISSA),
Via Bonomea 265,
34136 Trieste,
Italy}
\email{jenskaad@hotmail.com}

\author{Ryszard Nest}
\address{Department of Mathematical Sciences,
Universitetsparken 5,
DK-2100 Copenhagen {\O},
Denmark}
\email{rnest@math.ku.dk}
%
%
%
\thanks{The second author is partially supported by the Danish National Research Foundation (DNRF) through the Centre for Symmetry and Deformation.}

\begin{abstract}
We investigate the analytic properties of torsion isomorphisms (determinants) of mapping cone triangles of Fredholm complexes. Our main tool is a generalization to Fredholm complexes of the perturbation isomorphisms constructed by R. Carey and J. Pincus for Fredholm operators. A perturbation isomorphism is a canonical isomorphism of determinants of homology groups associated to a finite rank perturbation of Fredholm complexes. The perturbation isomorphisms allow us to establish the invariance properties of the torsion isomorphisms under finite rank perturbations. We then show that the perturbation isomorphisms provide a holomorphic structure on the determinant lines over the space of Fredholm complexes. Finally, we establish that the torsion isomorphisms and the perturbation isomorphisms provide holomorphic sections of certain determinant line bundles.
\end{abstract}

\maketitle
\tableofcontents

\section{Introduction}
One of the fundamental objects among the regulators of algebraic K-theory are the invariants of $K_*^{\T{alg}}$ of the quotients $\sL(H)/\sL^p(H)$ of the algebra of bounded operators on a Hilbert space $H$ modulo the ideal of $p$-summable operators, \cite{CoKa:CMF}. The fundamental object in the case of $K_2^{\T{alg}}$ is the Quillen line bundle of determinants of Fredholm operators, \cite{Qui:HAK,Qui:DCR}. R. Carey and J. Pincus introduced a canonical section of the pullback of the Quillen determinant line bundle to the space $\{(S,T) \in \sL(H)^2 \mid S,T \T{ Fredholm },\ S-T\in \sL^1(H) \}$, the so called perturbation vectors, \cite{CaPi:PV}. Their construction plays a key role in the study of the behaviour of L. Brown's determinant invariant on $K_2^{\T{alg}}(\sL(H)/\sL^1(H))$ and of determinants of Toeplitz operators, \cite{CaPi:JTS,Br:OAK}. In particular, it allowed them to prove a Szeg\"o limit theorem for Toeplitz operators with non-vanishing winding numbers, \cite{CaPi:STT,Wid:TAT}.
%

The extension of the results of Carey and Pincus to several commuting operators requires the generalization of the notion of a perturbation vector, {\em a perturbation isomorphism}, to the context of finite rank perturbations of Fredholm complexes. \footnote{Carey and Pincus work with trace class perturbations of Fredholm operators. The results in this paper can be extended to this setting, but it is not needed for the applications we have in mind.} 
%

\medskip

To be more concrete, let $X = X_+ \op X_-$ be a $\zz/2\zz$-graded vector space and let $d,\de : X \to X$ be odd maps such that $D = (X,d)$ and $E = (X,\de)$ are Fredholm complexes, i.e. $d^2 = 0 = \de^2$ and the homology groups $H_{\pm}(D)$ and $H_{\pm}(E)$ are finite dimensional. Suppose moreover that $d - \de : X \to X$ has finite rank. In this situation, we construct a canonical isomorphism
\[
P(E,D) : \T{det}(H_+(D)) \ot \T{det}(H_-(D))^* \to \T{det}(H_+(E)) \ot \T{det}(H_-(E))^*
\]
where $\T{det}(V) = \La^{\T{top}}(V)$ for any finite dimensional vector space $V$. We refer to this map as the perturbation isomorphism of the finite rank perturbation $D \to E$. Let us note that the perturbation isomorphism does not only depend on the homology groups but also on the differentials defining the chain complexes.

As an application of this construction, we obtain a holomorphic line bundle structure on determinants of bounded Fredholm complexes. This generalizes the construction due to Quillen of a holomorphic line bundle structure on determinants of Fredholm operators \footnote{The paper, \cite{Fre:DLB}, claims a similar result but unfortunately we have been unable to understand the proof.}, \cite{Qui:HAK}.
\medskip

\noindent A central topic of this paper is the investigation of the analytic properties of torsion isomorphisms of mapping cone triangles of Fredholm complexes. The above perturbation isomorphisms play a key role in formulating and proving these properties. Our torsion isomorphisms are defined as follows:

Suppose that $D^X = (X,d^X)$ and $D^Y = (Y,d^Y)$ are two $\zz/2\zz$-graded Fredholm complexes and that $A : D^X \to D^Y$ is a chain map. We denote by $C^A = (Y \op X[-1], d^A)$ the mapping cone of $A$ which fits into the following exact triangle:
\[
\De_A : \xymatrix{
D^X  \ar[rr]^A & & D^Y \ar[dl] \\
& C^A \ar[ul]^{[-1]} & 
}
\]
We will use the notation $\T{det}(D^X) = \T{det}(H_+(D^X)) \ot \T{det}(H_-(D^X))^*$ for the determinant of the $\zz/2\zz$-graded Fredholm complex $D^X$. The torsion isomorphism of the mapping cone triangle $\De_A$ is a canonical isomorphism 
\[
T(\De_A) : \T{det}(D^Y) \to \T{det}(D^X) \ot \T{det}(C^A)
\]
see Definition \ref{d:torisocha}. Up to a notorious sign-problem, this is a straightforward generalization of the canonical isomorphism
\[
T(V^*) : \T{det}(V^2) \to \T{det}(V^1) \ot \T{det}(V^3)
\]
associated to any short exact sequence $0 \to V^1 \to V^2 \to V^3 \to 0$ of finite dimensional vector spaces.
\bigskip

Our first main result describes the behaviour of torsion isomorphisms under finite rank perturbations:

On top of the above data, suppose that $E^X = (X,\de^X)$ and $E^Y = (Y,\de^Y)$ are $\zz/2\zz$-graded Fredholm complexes which are linked by a chain map $B : E^X \to E^Y$.

\begin{thm*}[{\bf Invariance under perturbations}]
Suppose that the differences $A - B : X \to Y$, $d^X - \de^X : X \to X$, and $d^Y - \de^Y : Y \to Y$ are of finite rank. Then the following diagram commutes,
\[
\begin{CD}
\T{det}(D^Y) @>{T(\De_A)}>> \T{det}(D^X)\ot \T{det}(C^A) \\
@V{P(E^Y,D^Y)}VV @VV{P(E^X,D^X) \ot P(C^B,C^A)}V \\
\T{det}(E^Y) @>>{T(\De_B)}> \T{det}(E^X)\ot \T{det}(C^B)
\end{CD}
\]
\end{thm*}

Suppose now that $X = \{X_n\}_{n \in \zz}$ is a fixed family of separable Hilbert spaces with only finitely many of them non-zero. Consider the space $\sC$ of $\zz$-graded Fredholm complexes on $X$ with all differentials being bounded operators. Furthermore, let $\sM$ denote the pull back
\[
\sM := \big\{ (D,E) \, | \, D = (X,d), E = (X,\de) \in \sC \T{ and } d - \de \T{ finite rank } \big\} 
\]

\begin{dfn}\label{d:holstuem}
We give $\sM$ the topology inherited from the norm $(D,E) \mapsto \|d\| + \|\de\| + \|d - \de\|_1$, where $\| \cd \|$ denotes the operator norm and $\| \cd \|_1$ denotes the trace norm. A map $\al : \sM \to \cc$ is \emph{holomorphic} when the pull back $\al \ci \phi : U \to \cc$ is holomorphic for any open subset $U \su \cc$ and any holomorphic map $\phi : U \to \sM$.
\end{dfn}

The analytic properties of the perturbation isomorphisms can now be expressed as follows:

\begin{thm*}[{\bf Analyticity of perturbation isomorphisms}]
The perturbation isomorphism
\[
(D,E) \mapsto P(D,E) 
\]
is a holomorphic section of the holomorphic line bundle over $\sM$ with fibers $\T{Hom}\big(\T{det}(E),\T{det}(D)\big)$.
\end{thm*}

Consider the space of triples
\[
\sX := \big\{ ( D, E, A) \, | \, D,E \in \sC \T{ and } A : D \to E \T{ bounded chain map}\big\}
\]

\begin{dfn} \label{d:holstuex}
We give $\sX$ the topology inherited from the \emph{operator norm}. A map $\al : \sX \to \cc$ is \emph{holomorphic} when the pull back $\al \ci \phi : U \to \cc$ is holomorphic for any open subset $U \su \cc$ and any holomorphic map $\phi : U \to \sX$.
\end{dfn}

The analytic properties of the torsion isomorphisms are then given by the following result:

\begin{thm*}[{\bf Analyticity of torsion isomorphisms}]
The torsion isomorphism 
\[
(D,E,A) \mapsto T(\De_A)
\]
is a holomorphic section of the holomorphic line bundle over $\sX$ with fibers $\T{Hom}\big( \T{det}(E), \T{det}(D) \ot \T{det}(C^A)\big)$.
\end{thm*}

\noindent The structure of the paper is as follows:
\medskip

\noindent In Section \ref{s:toriso} we define the torsion isomorphism of an exact triangle of $\zz/2\zz$-graded Fredholm complexes. 
\medskip

\noindent In Section \ref{s:periso} we construct the perturbation isomorphisms for finite rank perturbations of $\zz/2\zz$-graded Fredholm complexes and prove their basic algebraic properties. In particular, they satisfy relations of symmetry and transitivity, to wit:
\[
P(D^3,D^2)P(D^2,D^1) = P(D^3,D^1) \q \T{and} \q
P(D^2,D^1) = P(D^1,D^2)^{-1}
\]
\medskip

\noindent In Section \ref{s:speper} we compute the perturbation isomorphisms explicitly in a special case where they do not depend on the differentials but only on the homology groups.
\medskip

\noindent In Section \ref{s:permaptri} we study the perturbation isomorphisms associated to finite rank perturbations of mapping cone triangles. In particular, we prove the invariance under perturbations theorem stated above.
\medskip

\noindent In Section \ref{s:parperiso} we consider two holomorphic families of exact bounded Fredholm chain complexes such that the difference of differentials is holomorphic in trace norm. In this case, the perturbation isomorphisms define an invertible function, with values in the complex numbers, and we prove that it is holomorphic.
\medskip

\noindent In Section \ref{s:loctridet} we construct the holomorphic structure of the determinant line bundle on the space of Fredholm complexes.
\medskip

\noindent In Section \ref{s:anaperiso} we generalize the result of Section \ref{s:parperiso} and prove that the perturbation isomorphisms are holomorphic in the sense of Definition \ref{d:holstuem}.
\medskip

\noindent Finally, in Section \ref{s:anatoriso} we prove the main result, which says that the torsion isomorphism is in fact holomorphic in the sense of Definition \ref{d:holstuex}.

\section{Torsion isomorphisms}\label{s:toriso}
In this preliminary section we have collected some algebraic constructions which will be needed throughout this paper. 

In the first subsection we deal with determinants and torsion isomorphisms of finite dimensional vector spaces. The constructions are well known but we have to be very careful with the signs appearing in the various definitions, see \cite{KnMu:PMS}. In fact, our signs are \emph{not} the standard signs. Instead of using the usual torsion isomorphism $\T{det}(V^2) \to \T{det}(V^1) \ot \T{det}(V)$ of a short exact sequence
\[
\begin{CD}
0 @>>> V^1 @>>> V^2 @>>> V @>>> 0
\end{CD}
\]
as our corner stone, we have found it necessary to multiply this map by the extra sign $(-1)^{\T{dim}(V^1) \cd \T{dim}(V)}$. This sign corresponds to interchanging the factors $\T{det}(V^1)$ and $\T{det}(V)$ in the above tensor product.

In the second subsection we extend our constructions to chain complexes of (infinite dimensional) vector spaces. The chain complexes which we will consider are \emph{Fredholm} in the sense that they have finite dimensional homology groups. The passage to chain complexes described here is very similar to Breuning's construction of determinant functors on triangulated categories using cohomological functors, see \cite{Bre:DTC}.
\bigskip

Throughout this section $V$ and $W$ will be vector spaces over a field $\ff$. The vector space of linear maps from $V$ to $W$ will be denoted by $\sL(V,W)$ and the subspace of linear maps of finite rank will be denoted by $\sF(V,W) \su \sL(V,W)$.

Any linear map $A : V \to W$ has a \emph{pseudo-inverse} $A^\da : W \to V$. This is a linear endomorphism such that $A A^\da : W \to W$ and $1 - A^\da A : V \to V$ are idempotents with images $\T{Im}(AA^\da) = \T{Im}(A)$ and $\T{Im}(1 - A^\da A) = \T{Ker}(A)$, where $\T{Ker}(A)$ denotes the kernel of $A$ and $\T{Im}(A)$ denotes the image of $A$. Notice that a pseudo-inverse is \emph{not} unique.

A linear map $A : V \to W$ is \emph{Fredholm} when the kernel and the cokernel are finite dimensional over $\ff$. In this case the \emph{index} of $A$ is the difference $\T{Ind}(A) := \T{dim}\big( \T{Ker}(A) \big) - \T{dim}\big( \T{Coker}(A) \big)$ of dimensions.

\subsection{Torsion of vector spaces}\label{ss:torsio}
Let $n \in \nn_0$ and suppose that $V$ is $n$-dimensional as a vector space over $\ff$. The \emph{determinant of $V$} is the one-dimensional vector space defined as the top part of the exterior algebra over $V$, thus $\T{det}(V):= \La^n(V)$. We will often apply the notation $|V|$ for the determinant of $V$. Any isomorphism $\si : V \to W$ of $n$-dimensional vector spaces induces an isomorphism $|\si| := \T{det}(\si) : |V| \to |W|$ of determinant lines. This association is functorial, thus $\T{det}(\si_2) \ci \T{det}(\si_1) = \T{det}(\si_2 \ci \si_1)$ and $\T{det}(1) = 1$. For each non-zero vector $v \in |V|$ there is a unique dual vector $v^* \in |V|^*$ such that $v^*(v) = 1$.

The \emph{degree map} $\ep : \La(V) \to \nn_0$ on the exterior algebra over $V$ is defined on homogeneous elements by $\ep : v_1 \wlw v_k \mapsto k$.

Let $\ga : V \to V$ be an automorphism with $\ga^2 = 1$. Let $V^+$ and $V^-$ denote the eigenspaces for $\ga$ associated with the eigenvalues $1$ and $-1$ respectively. The \emph{determinant} of the $\zz/2\zz$-graded vector space $V$ is the line $|V| := \T{det}(V) := \T{det}(V^+) \ot \T{det}(V^-)^*$, where the ``$*$'' indicates that we take the dual vector space and ``$\ot$'' is the tensor product of vector spaces over $\ff$.
\bigskip

Consider a six term exact sequence of finite dimensional vector spaces
\[
\G V : \q \begin{CD}
V^1_+ @>{\pa_+}>> V_+^2 @>{i_+}>> V_+ \\
@A{p_-}AA & & @VV{p_+}V \\
V_- @<<{i_-}< V^2_- @<<{\pa_-}< V^1_-
\end{CD}
\]

Apply the notation
\[
\begin{split}
& (V_+^1)_{(0)} := \T{Ker}(\pa_+) \q (V_+^2)_{(0)} := \T{Ker}(i_+) \q (V_+)_{(0)} := \T{Ker}(p_+) \q \T{and} \\
& (V_-^1)_{(0)} := \T{Ker}(\pa_-) \q (V_-^2)_{(0)} := \T{Ker}(i_-) \q (V_-)_{(0)} := \T{Ker}(p_-)
\end{split}
\]

Choose subspaces
\begin{equation}\label{eq:subtor}
\begin{split}
& (V_+^1)_{(1)} \su V_+^1 \q (V_+^2)_{(1)} \su V_+^2 \q (V_+)_{(1)} \su V_+ \q \T{and} \\
& (V_-^1)_{(1)} \su V_-^1 \q (V_-^2)_{(1)} \su V_-^2 \q (V_-)_{(1)} \su V_-
\end{split}
\end{equation}
such that the following vector space decompositions hold,
\[
\begin{split}
& V_+^1 = (V_+^1)_{(0)} \dop (V_+^1)_{(1)} \q V_+^2 = (V_+^2)_{(0)} \dop (V_+^2)_{(1)} \q V_+ = (V_+)_{(0)} \dop (V_+)_{(1)} \\
& V_-^1 = (V_-^1)_{(0)} \dop (V_-^1)_{(1)} \q V_-^2 = (V_-^2)_{(0)} \dop (V_-^2)_{(1)} \q V_- = (V_-)_{(0)} \dop (V_-)_{(1)}
\end{split}
\]
Here the notation ``$\dop$'' refers to the span of two subspaces with trivial intersection.

\begin{dfn}\label{d:toriso}
The \emph{torsion isomorphism} of $\G V$ is the isomorphism
\[
T(\G V) : |V_+^2| \ot |V_-^2|^* \to (|V_+^1| \ot |V_-^1|^*) \ot (|V_+| \ot |V_-|^*)
\]
defined by
\[
\begin{split}
T(\G V) & : (\pa_+(t_+^1) \we t_+^2) \ot (\pa_-(t_-^1) \we t_-^2)^* \mapsto \\
& \q (-1)^{\mu(\G V)} \cd (p_-(t_-) \we t_+^1) \ot (p_+ (t_+) \we t_-^1)^* \\ 
& \qqq \qq \ot (i_+(t_+^2) \we t_+) \ot (i_-(t_-^2) \we t_-)^*
\end{split}
\]
for all non-zero vectors
\[
\begin{split}
& t_+^1 \in \big| (V^1_+)_{(1)} \big| \q t_+^2 \in \big| (V_+^2)_{(1)} \big| \q t_+ \in \big| (V_+)_{(1)} \big| \\
& t_-^1 \in \big| (V^1_-)_{(1)} \big| \q t_-^2 \in \big| (V_-^2)_{(1)} \big| \q t_- \in \big| (V_-)_{(1)} \big|
\end{split}
\]
The sign exponent is defined by
\[
\begin{split}
\mu(\G V) & := (\ep(t_+^2) + 1) \cd \big( \ep(t_-^1) + \ep(t_+^1) \big) + \ep(t_-^1) \cd \big( \ep(t_+) + \ep(t_-)\big) \\ 
& \q + \ep(t_-) \cd \big(\ep(t_+^2) + \ep(t_-^2)\big) + \ep(t_+) \in \nn_0.
\end{split}
\]
\end{dfn}

It is a consequence of \cite[Lemma 2.1.3]{Kaa:JTS} that the torsion isomorphism does not depend on the choice of the subspaces in \eqref{eq:subtor}.

Notice that the torsion isomorphism is a non-zero vector in a one-dimensional vector space over $\ff$ and not a non-zero element in the field $\ff$.

\subsection{Torsion of chain complexes}
Let $D^1 := (X^1,d^1)$ and $D^2 := (X^2,d^2)$ be bounded chain complexes over the field $\ff$. Let $A : D^1 \to D^2$ be a chain map.

\begin{dfn}
The \emph{mapping cone} of $A$ is the chain complex $C^A$ defined by the chains $C^A_j := X^2_j \op X^1_{j-1}$ and the differentials
\[
d^A_j := \ma{cc}{d^2_j & A_{j-1} \\ 0 & - d^1_{j-1}} : C^A_j \to C^A_{j-1}
\]
for all $j \in \zz$.
\end{dfn}

The notation $T D^1$ will refer to the shifted chain complex. The chains of $TD^1$ are defined by $(TX^1)_j := X^1_{j-1}$ and the differentials are defined by $(Td^1)_j := -d_{j-1}^1 : X^1_{j-1} \to X^1_{j-2}$ for all $j \in \zz$.

\begin{dfn}
The \emph{mapping cone triangle} of the chain map $A : D^1 \to D^2$ is the sequence
\[
\begin{CD}
\G D : D^1 @>{A}>> D^2 @>{i}>> C^A @>{p}>> T D^1,
\end{CD}
\]
of chain maps, where $i_j : X^2_j \to X^2_j \op X^1_{j-1}$ denotes the inclusion and $p_j : X^2_j \op X^1_{j-1} \to X^1_{j-1}$ is the projection.
\end{dfn}

The mapping cone triangle gives rise to a long exact sequence of homology groups,
\begin{equation}\label{eq:lonmapcon}
\begin{CD}
\ldots @<{A_{j-1}}<< H_{j-1}(D^1) @<{p_j}<< H_j(C^A) @<{i_j}<< H_j(D^2) @<{A_j}<< H_j(D^1) @<{p_{j+1}}<< \ldots 
\end{CD}
\end{equation}

For a bounded chain complex $D$, the notation $H_+(D)$ and $H_-(D)$ refers to the even and odd homology groups respectively, thus
\[
H_+(D) := \op_{k \in \zz} H_{2k}(D) \q H_-(D) := \op_{k \in \zz} H_{2k-1}(D).
\]
Similarly, for a chain map $B : D \to \De$ of bounded chain complexes, let $B_+ : H_+(D) \to H_+(\De)$ and $B_- : H_-(D) \to H_-(\De)$ denote the linear maps given by
\[
B_+(\{x_{2k}\} ) := \{ B_{2k}(x_{2k})\} \q B_-(\{x_{2k-1}\}) := \{B_{2k-1}(x_{2k-1}) \}.
\]

The long exact sequence in \eqref{eq:lonmapcon} then gives rise to a six term exact sequence
\begin{equation}
H(\G D) : \begin{CD}
H_+(D^1) @>{A_+}>> H_+(D^2) @>{i_+}>> H_+(C^A) \\
@A{p_-}AA & & @VV{p_+}V \\
H_-(C^A) @<<{i_-}< H_-(D^2) @<<{A_-}< H_-(D^1)
\end{CD}
\end{equation}
of even and odd homology groups.

\begin{dfn}\label{d:chafre}
A bounded chain complex $D$ is \emph{Fredholm} when all the homology groups are finite dimensional.

The \emph{index} of a Fredholm complex $D$ is the integer
\[
\T{Ind}(D) := \T{dim}(H_+(D)) - \T{dim}(H_-(D))
\]

The \emph{determinant} of a Fredholm complex $D$ is the one-dimensional vector space
\[
|D| := \T{det}(D) := |H_+(D)| \ot |H_-(D)|^*
\]
\end{dfn}

\begin{dfn}\label{d:torisocha}
Suppose that $D^1$ and $D^2$ are Fredholm complexes. The \emph{torsion isomorphism} of the mapping cone triangle $\G D$ is defined as the torsion isomorphism of the associated six term exact triangle $H(\G D)$. It is denoted by
\[
T(\G D) := T(H(\G D)) : |D^2| \to |D^1| \ot |C^A|.
\]
\end{dfn}

We remark that it follows from the long exact sequence in \eqref{eq:lonmapcon} that $C^A$ is a Fredholm complex whenever $D^1$ and $D^2$ are Fredholm complexes.

\section{Perturbation isomorphisms}\label{s:periso}
\emph{Throughout this section $V$ and $W$ will be (infinite dimensional) vector spaces over a field $\ff$.}

Consider two Fredholm operators $A$ and $B : V \to W$ and suppose that their difference $A - B$ has finite rank. In the paper, \cite{CaPi:PV}, R. Carey and J. Pincus construct a \emph{canonical} isomorphism of the determinant lines associated with $A$ and $B$. This isomorphism is referred to as the \emph{perturbation isomorphism} from $A$ to $B$ and it is denoted by
\[
P(B,A) : \T{det}(\T{Ker}(A)) \ot \T{det}(\T{Coker}(A))^* \to \T{det}(\T{Ker}(B)) \ot \T{det}(\T{Coker}(B))^*
\]

In this section we will generalize the construction of R. Carey and J. Pincus to the framework of Fredholm complexes. To be more precise, we will consider two $\zz/2\zz$-graded Fredholm complexes 
\[
\begin{CD}
D : X_+ @>{d_+}>> X_- @>{d_-}>> X_+ \q \T{and} \q \De : X_+ @>{\de_+}>> X_- @>{\de_-}>> X_+   
\end{CD}
\]
such that the differences $d_+ - \de_+ : X_+ \to X_-$ and $d_- - \de_- : X_- \to X_+$ both have finite rank. It is then our aim to construct a canonical isomorphism of determinant lines,
\[
P(\De,D) : \big| H_+(D) \big| \ot \big| H_-(D) \big|^* \to \big| H_+(\De) \big| \ot \big| H_-(\De) \big|^*
\]
Notice here that $H_+(D) := \T{Ker}(d_+)/\T{Im}(d_-)$ and $H_-(D) := \T{Ker}(d_-) / \T{Im}(d_+)$ denote the homology groups of a $\zz/2\zz$-graded Fredholm complex $D = (X,d)$.

The organization of this section is as follows: In the first subsection we review the construction of R. Carey and J. Pincus. In the second subsection we define the perturbation isomorphism in the $\zz/2\zz$-graded context. In the fourth subsection we prove that our new perturbation isomorphism is well-defined. This relies on the computation of various determinants and this computation is carried out in the third subsection. In the last subsection we prove that our perturbation isomorphisms are symmetric and transitive.
\bigskip

A linear endomorphism $T : V \to V$ is of \emph{determinant class} when it is invertible and when $1 - T$ has finite rank. Let $T : V \to V$ be of determinant class and let $E \su V$ be a finite dimensional invariant subspace for $T$ with $\T{Im}(1 - T) \su E$. The \emph{determinant} of $T$ is then defined as the invertible number
\[
\T{det}(T) := \T{det}(T|_E) \in \ff^*
\]
where $\T{det}(T|_E)$ denotes the determinant of the restriction of $T$ to $E$.

\subsection{Perturbations of Fredholm operators}\label{ss:perfre}
\emph{Throughout this subsection $A$ and $B : V \to W$ will be Fredholm operators such that the difference $A - B$ has finite rank.}

Let us recall the construction due to R. Carey and J. Pincus of a canonical isomorphism
\[
P(B,A) : \T{det}(\T{Ker}(A)) \ot \T{det}(\T{Coker}(A))^* \to \T{det}(\T{Ker}(B)) \ot \T{det}(\T{Coker}(B))^*
\]
see \cite[\S 3]{CaPi:PV}. This isomorphism is referred to as the \emph{perturbation isomorphism} from $A$ to $B$.

Choose pseudo-inverses $A^\da \T{ and } B^\da : W \to V$ of the Fredholm operators $A$ and $B$. These pseudo-inverses are again Fredholm operators with indices given by
\begin{equation}\label{eq:indiden}
\T{Ind}(A^\da) = - \T{Ind}(A) = - \T{Ind}(B) = \T{Ind}(B^\da).
\end{equation}
Remark that $\T{Ind}(A) = \T{Ind}(B)$ since $A - B$ has finite rank by assumption.

Apply the notation
\[
\begin{split}
& P_A := 1 - A^\da A \, , \, P_B := 1 - B^\da B : V \to V \q  \T{and} \\
& Q_A := 1 - A A^\da \, , \, Q_B := 1 - B B^\da : W \to W
\end{split}
\]
for the various idempotents associated with the pseudo-inverses $A^\da$ and $B^\da$. Notice that all these idempotents are of finite rank. Indeed,
\[
\begin{split}
& \T{Im}(P_A) = \T{Ker}(A) \q \T{Im}(P_B) = \T{Ker}(B) \q \T{and} \\
& \T{Im}(Q_A) = \T{Ker}(A^\da) \q \T{Im}(Q_B) = \T{Ker}(B^\da).
\end{split}
\]

\emph{Suppose that the common index of $A$ and $B$ is less than or equal to zero.} Choose a linear map $L : \T{Ker}(A) \to W$ such that the composition
\[
\begin{CD}
\T{Ker}(A) @>{L}>> W @>>> \T{Coker}(A)
\end{CD}
\]
is injective. Choose a linear map $M : \T{Ker}(B^\da) \to V$ such that the composition
\[
\begin{CD}
\T{Ker}(B^\da) @>{M}>> V @>>> \T{Coker}(B^\da)
\end{CD}
\]
is surjective. It then follows from the index identity \eqref{eq:indiden} that the vector spaces $\T{Ker}(M)$ and $W/\big( \T{Im}(A) + \T{Im}(L) \big)$ are isomorphic. Choose a linear map $N : \T{Ker}(M) \to W$ such that the composition
\[
\begin{CD}
\T{Ker}(M) @>N>> W @>>> W/\big( \T{Im}(A) + \T{Im}(L) \big)
\end{CD}
\]
is an isomorphism. Extend $N$ to $\T{Ker}(B^\da)$ by letting it equal zero on some algebraic complement of $\T{Ker}(M) \su \T{Ker}(B^\da)$. The data $(L,M,N)$ will be referred to as a \emph{perturbation triple}.

The perturbation triple is defined in such a way that the linear endomorphism
\[
\Si := (A + L P_A)(B^\da + M Q_B) + N Q_B : W \to W.
\]
is an isomorphism of determinant class.
%
%

\begin{dfn}\label{d:defpermin}
Suppose that the common index of $A$ and $B$ is less than or equal to zero. The \emph{perturbation isomorphism} from $A$ to $B$ is the isomorphism
\[
P(B,A) : \T{det}(\T{Ker}(A)) \ot \T{det}(\T{Coker}(A))^* \to \T{det}(\T{Ker}(B)) \ot \T{det}(\T{Coker}(B))^*
\]
defined by
\[
P(B,A) : s \ot (Ls \we N t_1)^* \mapsto \T{det}(\Si)^{-1} \cd \big( Mt_0 \ot (t_0 \we t_1)^* \big)
\]
for all non-trivial vectors $s \in \T{det}(\T{Ker}(A))$, $t_1 \in \T{det}(\T{Ker}(M))$ and $t_0 \in \T{det}(\T{Ker}(N))$.
\end{dfn}

\emph{Suppose that the common index of $A$ and $B$ is greater than or equal to zero.} Choose a linear map $L : \T{Ker}(A) \to W$ such that the composition
\[
\begin{CD}
\T{Ker}(A) @>{L}>> W @>>> \T{Coker}(A)
\end{CD}
\]
is surjective. Choose a linear map $M : \T{Ker}(B^\da) \to V$ such that the composition
\[
\begin{CD}
\T{Ker}(B^\da) @>{M}>> V @>>> \T{Coker}(B^\da)
\end{CD}
\]
is injective. As above we may then choose a linear map $N : \T{Ker}(L) \to V$ such that the composition
\[
\begin{CD}
\T{Ker}(L) @>{N}>> V @>>> V/ \big( \T{Im}(M) + \T{Im}(B^\da) \big)
\end{CD}
\]
is an isomorphism. The map $N$ is extended to $\T{Ker}(A)$ by letting it equal zero on a vector space complement of $\T{Ker}(L) \su \T{Ker}(A)$. As above, we refer to the data $(L,M,N)$ as a \emph{perturbation triple}.
%

The isomorphism of determinant class is in this case given by
\[
\Si = (B^\da + M Q_B)(A + L P_A) + N P_A : V \to V.
\]

\begin{dfn}\label{d:defperplu}
Suppose that the common index of $A$ and $B$ is greater than or equal to zero. The \emph{perturbation isomorphism} from $A$ to $B$ is the isomorphism
\[
P(B,A) : \T{det}(\T{Ker}(A)) \ot \T{det}(\T{Coker}(A))^* \to \T{det}(\T{Ker}(B)) \ot \T{det}(\T{Coker}(B))^*
\]
defined by
\[
P(B,A) : s_0 \we s_1 \ot (Ls_1)^* \mapsto \T{det}(\Si)^{-1} \cd \big( (Ns_0 \we Mt) \ot t^* \big)
\]
for all non-trivial vectors $s_0 \in \T{det}(\T{Ker}(L))$, $s_1 \in \T{det}(\T{Ker}(N))$ and $t \in \T{det}(\T{Coker}(B))$.
\end{dfn}

It is proved in \cite[Theorem 11]{CaPi:PV} that the perturbation isomorphism is \emph{independent} of the choice of pseudo-inverses of $A$ and $B$ and of the choice of perturbation triple $(L,M,N)$.

The main algebraic properties of the perturbation isomorphism are proved in \cite[Theorem 12]{CaPi:PV}. The result is stated here for the convenience of the reader.

\begin{thm}\label{t:algprofre}
Let $A,B,C : V \to W$ be Fredholm operators and suppose that the differences $A - B$ and $B - C$ are of finite rank. Then the following holds:
\begin{enumerate}
\item $P(A,B) \ci P(B,A) = 1$.
\item $P(C,B) \ci P(B,A) = P(C,A)$.
\end{enumerate}
\end{thm}

\subsection{Perturbations of Fredholm complexes}\label{ss:perfrecom}
\emph{Throughout this subsection $D = (X,d)$ and $\De = (X,\de)$ will be $\zz/2\zz$-graded Fredholm complexes over $\ff$ such that the differences of differentials $d_+ - \de_+ : X_+ \to X_-$ and $d_- - \de_- : X_- \to X_+$ have \emph{finite rank}.} We will say that $\De$ is a \emph{finite rank perturbation} of $D$.

Out of this data, we are interested in constructing a \emph{canonical} isomorphism
\[
P(\De,D) : |H_+(D)| \ot |H_-(D)|^* \to |H_+(\De)| \ot |H_-(\De)|^*
\]
of determinant lines. We will sometimes apply the short notation $|D| := |H_+(D)| \ot |H_-(D)|^*$ for the determinant of a $\zz/2\zz$-graded Fredholm complex $D$.

Choose a pseudo-inverse of the $\zz/2\zz$-graded Fredholm complexes $D$,
\[
\begin{CD}
D^\da : X_- @>d_+^\da>> X_+ @>d_-^\da>> X_- 
\end{CD} 
\]
This means that $D^\da$ is a $\zz/2\zz$-graded Fredholm complex and that $d_+^\da$ and $d_-^\da$ are pseudo-inverses of $d_+ : X_+ \to X_-$ and $d_- : X_- \to X_+$, respectively.

Apply the notation
\[
\si_+ := d_+ + d_-^\da : X_+ \to X_- \q \T{and} \q \si_- := d_- + d_+^\da : X_- \to X_+
\]
Remark that $\si_+$ and $\si_-$ are Fredholm operators. Furthermore, $\si_-$ is a pseudo-inverse of $\si_+$. The associated idempotents are denoted by
\[
P_+ := 1 - \si_- \si_+ : X_+ \to X_+ \q \T{and} \q P_- := 1 - \si_+ \si_- : X_- \to X_-
\]
Remark that the quotient maps induce isomorphisms
\[
\T{Im}(P_+) \cong H_+(D) \q \T{and} \q \T{Im}(P_-) \cong H_-(D)
\]

\begin{lemma}\label{findiff}
There exists a pseudo-inverse
\[
\begin{CD}
\De^\da : X_- @>\de_+^\da>> X_+ @>\de_-^\da>> X_-
\end{CD}
\]
of the $\zz/2\zz$-graded Fredholm complex $\De$ such that $\De^\da$ is a finite rank
perturbation of $D^\da$.
\end{lemma}
\begin{proof}
Start by choosing a pseudo-inverse $\de_+^\da : X_- \to X_+$ of $\de_+ : X_+ \to X_-$ such that the
difference
\[
\de_+^\da - d_+^\da : X_- \to X_+
\]
is an operator finite rank. Remark that care should be taken at this
point. Indeed, since the operator $\de_+$ is not necessarily a Fredholm
operator, there might exist pseudo-inverses of $\de_+$ which do not satisfy
the above finite rank condition.

Next, choose a pseudo-inverse $\de_-^\da : X_+ \to X_-$ of $\de_- : X_- \to X_+$ such that
\[
\begin{CD}
\De^\da : X_- @>\de_+^\da>> X_+ @>\de_-^\da>> X_-
\end{CD}
\]
becomes a $\zz/2\zz$-graded Fredholm complex. We need to prove that the
difference
\[
\de_-^\da - d_-^\da : X_+ \to X_-
\]
is an operator of finite rank. However, this is true if and only if the
difference
\[
(\de_+ + \de_-^\da) - (d_+ + d_-^\da) : X_+ \to X_-
\]
is of finite rank. To prove that this last difference is of finite rank, we note that $\de_+ + \de_-^\da$ is a pseudo-inverse of $\de_- + \de_+^\da$. Likewise, $d_+ + d_-^\da$ is a pseudo-inverse of $d_- + d_+^\da$. Recall also that $\de_- + \de_+^\da$ and $d_- + d_+^\da$ are Fredholm operators. The result now follows since the difference 
\[
(\de_- + \de_+^\da) - (d_- + d_+^\da) : X_- \to X_+
\]
is an operator of finite rank. Indeed, for any pair of Fredholm operators $A$ and $B : V \to W$ with $A - B \in \sF(V,W)$ we have that $A^\da - B^\da \in \sF(W,V)$ for any pair of pseudo-inverses $A^\da$ and $B^\da : W \to V$.
\end{proof}

Choose a pseudo-inverse of $\De$,
\[
\begin{CD}
\De^\da : X_- @>\de_+^\da>> X_+ @>\de_-^\da>> X_-
\end{CD} 
\]
which is a finite rank perturbation of $D^\da$.

Apply the notation
\[
\tau_+ := \de_+ + \de_-^\da : X_+ \to X_- \q \T{and} \q \tau_- := \de_- + \de_+^\da : X_- \to X_+
\]
for the associated Fredholm operators. The idempotents arising from this pair are denoted by
\[
Q_+ := 1 - \tau_- \tau_+ : X_+ \to X_+ \q \T{and} \q Q_- := 1 - \tau_+ \tau_- : X_- \to X_-
\]
The quotient maps then induce isomorphisms
\[
\T{Im}(Q_+) \cong H_+(\De) \q \T{and} \q \T{Im}(Q_-) \cong H_-(\De)
\]

\begin{dfn}\label{d:perisocom}
The \emph{perturbation isomorphism} of the finite rank perturbation $D \to \De$ of $\zz/2\zz$-graded Fredholm complexes is defined as the perturbation isomorphism of the finite rank perturbation of Fredholm operators $\si_+ \to \tau_+$, see Definition \ref{d:defpermin} and Definition \ref{d:defperplu}. The perturbation isomorphism is denoted by
\[
P(\De,D) : \big| H_+(D) \big| \ot \big| H_-(D) \big|^* \to \big| H_+(\De) \big| \ot \big| H_-(\De)\big|^*
\]
\end{dfn}
\bigskip

Remark that we have tacitly applied the canonical isomorphisms of determinant lines $|D| \cong |\T{Ker}(\si_+)| \ot |\T{Coker}(\si_+)|^*$ and $|\De| \cong |\T{Ker}(\tau_+)| \ot |\T{Coker}(\tau_+)|^*$ in the above definition. These isomorphisms come from the canonical isomorphism of vector spaces
\[
\begin{split}
& H_+(D) \cong \T{Ker}(\si_+) \q H_-(D) \cong \T{Coker}(\si_+) \q \T{and} \\
& H_+(\De) \cong \T{Ker}(\tau_+) \q H_-(\De) \cong \T{Coker}(\tau_+)
\end{split}
\]

We need to prove that the perturbation isomorphism is well-defined in this new context. That is, it should be independent of the choice of pseudo-inverses $D^\da$ and $\De^\da$. We will show that this is indeed the case in Subsection \ref{ss:ind}. The proof of this result requires some preliminary observations on the determinants of certain operators. These observations are made in the next subsection.

\subsection{The determinants of some linear operators}
Let $A,B : V \to V$ be endomorphisms of a vector space $V$. Suppose that their squares are trivial $A^2 = 0 = B^2$ and that their sum is a finite rank operator, $A + B \in \sF(V)$.

The operators $1 + A$ and $1+ B$ are then invertible operators. Furthermore, their product
\[
(1 + A)(1 + B) = 1 + A + B + AB = 1 + (A + B)(1 + B) : V \to V
\]
is of determinant class.

\begin{lemma}\label{l:trivnil}
\[
\T{det}\big( (1 + A)(1 + B) \big) = 1
\]
\end{lemma}
\begin{proof}
Notice that the image of $A + B \in \sF(V)$ is an invariant subspace for both of the invertible operators $1 + A$ and $1 + B$. Indeed, we have the identities
\[
\arr{ccc}{
(1 + A)(A + B) = (A + B)(1 + B) 
& \T{and} &
(1 + B)(A + B) = (A + B)(1 + A)
}
\]
It follows that the determinants in question can be expressed as a product of
determinants,
\[
\begin{split}
\T{det}\big( (1 + A)(1 + B) \big)
& = \T{det}\big( (1 + A)(1 + B)|_{\T{Im}(A + B)} \big) \\
& = \T{det}\big( (1 + A)|_{\T{Im}(A + B)} \big)
\cd \T{det}\big( (1 + B)|_{\T{Im}(A + B)} \big)
\end{split}
\]
But both of the determinants in this factorization are trivial since $A^2 = 0 = B^2$ by assumption.
\end{proof}

Consider now an arbitrary isomorphism $\Om : V \to V$ of determinant class. Furthermore, let $T : V \to W$ be an injective Fredholm operator and let $S : W \to V$ be a surjective Fredholm operator. Suppose that the
differences $TS - 1 : W \to W$ and $ST - 1 : V \to V$ are operators of finite rank. This implies that 
\[
\T{dim}(\T{Ker}(S)) = \T{Ind}(S) = - \T{Ind}(T) = \T{dim}(\T{Coker}(T))
\]
Choose a linear map $N : \T{Ker}(S) \to W$ such that the composition
\[
\begin{CD}
\T{Ker}(S) @>N>> W @>>> \T{Coker}(T)
\end{CD}
\]
becomes an isomorphism. Here the last map is the quotient map. Extend the linear map $N$ to a linear map $N : W \to W$ by letting it equal zero on some vector space complement of $\T{Ker}(S) \su W$.

\begin{lemma}\label{l:detequ}
The linear maps
\[
\Si := T S + N \, \T{ and } \, \wit \Si := T \Om S + N : W \to W
\]
are both of determinant class and their determinants satisfy the relation
\[
\T{det}(\wit \Si) = \T{det}(\Om) \cd \T{det}(\Si)
\]
\end{lemma}
\begin{proof}
The fact that $\Si$ and $\wit \Si$ are of determinant class follows immediately since the linear maps $TS - 1 : W \to W$, $\Om - 1 : V \to V$ and $N : W \to W$ are of finite rank.

In order to prove the identity of determinants we start by writing down an
inverse of $\Si$. To this end, choose pseudo-inverses
\[
T^\da : W \to V \q S^\da : V \to W \q N^\da : W \to W 
\]
such that 
\[
T^\da N = S N^\da = 0 \q \T{and} \q N S^\da = N^\da T = 0.
\]
The inverse of $\Si$ is then given by the formula
\[
\Si^{-1} = N^\da + S^\da T^\da : W \to W.
\]
This implies that
\[
\Si^{-1} \wit \Si = (N^\da + S^\da T^\da)( T \Om S + N) =  N^\da N + S^\da \Om S.
\]
In particular, we get that
\[
\T{det}(\Si^{-1} \wit \Si) 
= \T{det}\big( S^\da \Om S : \T{Im}(S^\da) \to \T{Im}(S^\da)\big)
= \T{det}(\Om)
\]
This proves the statement of the lemma.
\end{proof}

\subsection{Independence of pseudo-inverses}\label{ss:ind}
We are now ready to prove that our perturbation isomorphism is independent of the choices of pseudo-inverses. This is the subject of the present subsection.

Let us make the general assumption that the common index of our Fredholm complexes
\[
\T{Ind}(D) = \T{Ind}(\De) \leq 0
\]
is less than or equal to zero. The independence result for the perturbation isomorphism can be verified in a similar fashion when the common index is greater than or equal to zero.

Choose alternative pseudo-inverses
\[
\begin{CD}
D^* : X_- @>d_+^*>> X_+ @>d_-^*>> X_- 
\q \T{and} \q 
\De^* : X_- @>{\de_+^*}>> X_+ @>{\de_-^*}>> X_-
\end{CD}
\]
of the Fredholm complexes
\[
\begin{CD}
D : X_+ @>{d_+}>> X_- @>{d_-}>> X_+
\q \T{and} \q
\De : X_+ @>{\de_+}>> X_- @>{\de_-}>> X_+
\end{CD}
\]
As in Subsection \ref{ss:perfrecom} it is required that the differences
\[
d_+^* - \de_+^* : X_- \to X_+ \q \T{and} \q
d_-^* - \de_-^* : X_+ \to X_-
\]
are operators of finite rank.

The associated Fredholm operators are denoted by
\[
\begin{split}
& \wit \si_+ := d_+ + d_-^* : X_+ \to X_- \q \wit \si_- := d_- + d_+^* : X_- \to X_+ \q \T{and} \\
& \wit \tau_+ := \de_+ + \de_-^* : X_+ \to X_- \q \wit \tau_- := \de_- + \de_+^* : X_- \to X_+
\end{split}
\]
The idempotents of finite rank are denoted by
\[
\begin{split}
& \wit P_+ := 1 - \wit \si_- \wit \si_+ : X_+ \to X_+ \q
\wit P_- := 1 - \wit \si_+ \wit \si_- : X_- \to X_- \q \T{and} \\
& \wit Q_+ := 1 - \wit \tau_- \wit \tau_+ : X_+ \to X_+ \q
\wit Q_- := 1 - \wit \tau_+ \wit \tau_- : X_- \to X_-
\end{split}
\]

The relations between the various Fredholm operators is encoded in the following eight invertible maps,
\begin{equation}\label{eq:isodifcho}
\begin{split}
\Om_-^d & := 1 - d_- d_-^\da + d_-d_-^* : X_+ \to X_+ \q 
\La_-^d := 1 - d_-^\da d_- + d_-^* d_- : X_- \to X_- \\
\Om_+^d & := 1 - d_+ d_+^\da + d_+d_+^* : X_- \to X_- \q
\La_+^d := 1 - d_+^\da d_+ + d_+^* d_+ : X_+ \to X_+ \\
\Om_-^\de & := 1 - \de_- \de_-^\da + \de_-\de_-^* : X_+ \to X_+ \q
\La_-^\de := 1 - \de_-^\da \de_- + \de_-^* \de_- : X_- \to X_- \\
\Om_+^\de & := 1 - \de_+ \de_+^\da + \de_+\de_+^* : X_- \to X_- \q
\La_+^\de := 1 - \de_+^\da \de_+ + \de_+^* \de_+ : X_+ \to X_+
\end{split}
\end{equation}
Here the invertibility follows by noting that all the maps are of the form $1 + A$ where $A$ has trivial square. Remark that these invertible maps are not necessarily of determinant class.

\begin{lemma}\label{l:fredrela}
We have the relations
\[
\begin{split}
\La_-^d \si_+ \Om_-^d & = \wit \si_+ : X_+ \to X_- \q 
\La_+^d \si_- \Om_+^d = \wit \si_- : X_- \to X_+ \q \T{and} \\
\La_-^\de \tau_+ \Om_-^\de & = \wit \tau_+ : X_+ \to X_- \q
\La_+^\de \tau_- \Om_+^\de = \wit \tau_- : X_- \to X_+
\end{split}
\]
between the Fredholm operators associated with different choices of
pseudo-inverses.
\end{lemma}
\begin{proof}
We will only verify the first of these identities. The proof of the other
identities is similar. This first identity follows from the computation
\[
\begin{split}
\La_-^d \si_+ \Om_-^d
& = (1 - d_-^\da d_- + d_-^* d_-)(d_+ + d_-^\da)(1 - d_-d_-^\da + d_-d_-^*) \\
& = (d_+ + d_-^* d_- d_-^\da)(1 - d_- d_-^\da + d_- d_-^*) \\
& = d_+ + d_-^* d_- d_-^\da d_- d_-^* \\
& = \wit \si_+
\end{split}
\]
\end{proof}

It follows from the above lemma that the isomorphisms in \eqref{eq:isodifcho} induce isomorphisms
\begin{equation}\label{eq:omed}
\begin{split}
& \Om_-^d : \T{Ker}(\wit \si_+) \to \T{Ker}(\si_+) \qq \Om_+^d : \T{Ker}(\wit \si_-) \to \T{Ker}(\si_-) \\
& \La_-^d : \T{Coker}(\si_+) \to \T{Coker}(\wit \si_+) \q \La_+^d : \T{Coker}(\si_-) \to \T{Coker}(\wit \si_-) \q \T{and} \\
& \Om_-^{\de} : \T{Ker}(\wit \tau_+) \to \T{Ker}(\tau_+) \qq
\Om_+^{\de} : \T{Ker}(\wit \tau_-) \to \T{Ker}(\tau_-) \\
& \La_-^\de : \T{Coker}(\tau_+) \to \T{Coker}(\wit \tau_+) \q
\La_+^\de : \T{Coker}(\tau_-) \to \T{Coker}(\wit \tau_-)
\end{split}
\end{equation}
between the various kernels and cokernels of our Fredholm operators.
\bigskip

Let us now choose a perturbation triple $(L,M,N)$ for the finite rank perturbation $\si_+ \to \tau_+$ of Fredholm operators. Recall that the composition
\[
\begin{CD}
\T{Ker}(\si_+) @>{L}>> X_- @>>> \T{Coker}(\si_+)
\end{CD}
\] 
is an injective map and that the composition
\[
\begin{CD}
\T{Ker}(\tau_-) @>{M}>> X_+ @>>> \T{Coker}(\tau_-)
\end{CD}
\]
is a surjective map. Recall furthermore that the composition
\[
\begin{CD}
\T{Ker}(M) @>{N}>> X_- @>>> X_- / \big( \T{Im}(L) + \T{Im}(\si_+) \big)
\end{CD}
\]
is an isomorphism. The map $N : \T{Ker}(\tau_-) \to X_-$ is extended to $\T{Ker}(\tau_-)$ by letting it equal zero on a vector space complement of $\T{Ker}(M) \su \T{Ker}(\tau_-)$.

Define the linear maps
\[
\begin{split}
& \wit L := \La_-^d L \Om_-^d : \T{Ker}(\wit \si_+) \to X_- \\
& \wit M := \La_+^\de M \Om_+^\de : \T{Ker}(\wit \tau_-) \to X_+ \q \T{and} \\
& \wit N := \La_-^d N \Om_+^\de : \T{Ker}(\wit \tau_-) \to X_-
\end{split}
\]
by combining the isomorphisms in \eqref{eq:omed} with the perturbation triple $(L,M,N)$.

\begin{lemma}\label{l:perttripalt}
The triple of linear maps $(\wit L, \wit M, \wit N)$ is a perturbation triple for the finite rank perturbation $\wit \si_+ \to \wit \tau_+$.
\end{lemma}
\begin{proof}
This follows by construction, since $(L,M,N)$ is a perturbation triple for the finite rank perturbation $\si_+ \to \tau_+$.
\end{proof}

\begin{lemma}\label{l:kapequaleq}
The maps between the homology groups of $D$ and $\De$ induced by the perturbation triple $(L,M,N)$ and the perturbation triple $(\wit L,\wit M,\wit N)$ agree.
\end{lemma}
\begin{proof}
This follows by noting that the maps
\[
\begin{split}
& \Om_-^d \, \, , \, \, \Om_-^\de : X_+ \to X_+ \q
\Om_+^d \, \, , \, \, \Om_+^\de : X_- \to X_- \\
& \La_-^d \, \, , \, \, \La_-^\de : X_- \to X_- \q
\La_+^d \, \, , \, \, \La_+^\de : X_+ \to X_+
\end{split}
\]
all induce the identity map on the various homology groups in question. Indeed, this implies that we have the following identities
\[
\begin{split}
& L = \wit L : H_+(D) \to H_-(D) \q M = \wit M : H_-(\De) \to H_+(\De) \q \T{and} \\ 
& N = \wit N : H_+(D) \to H_-(\De)
\end{split}
\]
at the level of homology groups between the various maps induced by the perturbation triples.
\end{proof}

Apply the notation
\[
\wit \Si := (\wit \si_+ + \wit L \wit P_+)(\wit \tau_- + \wit M \wit Q_-) + \wit N \wit Q_-
: X_- \to X_-
\]
for the isomorphism of determinant class associated with the finite rank perturbation $\wit \si_+ \to \wit \tau_+$ and the perturbation triple $(\wit L, \wit M, \wit N)$.

Likewise, we have the isomorphism of determinant class
\[
\Si := (\si_+ + L P_+)(\tau_- + M Q_-) + N Q_- : X_- \to X_-
\]

\begin{lemma}\label{l:detequaleq}
The determinants $\T{det}(\Si)$ and $\T{det}(\wit \Si)$ coincide.
\end{lemma} 
\begin{proof}
Remark first that we have the identities
\[
P_+ \Om_-^d = P_+ \, \, , \, \, \wit P_+ \Om_-^d = \wit P_+ \q \T{and} \q
Q_- \Om_+^\de = Q_- \, \, , \, \, \wit Q_- \Om_+^\de = \wit Q_-
\]

Using Lemma \ref{l:fredrela} and the definition of the perturbation triple $(\wit L, \wit M, \wit N)$ we can then compute as follows,
\[
\begin{split}
\wit \Si
& = (\wit \si_+ + \wit L \wit P_+)(\wit \tau_- + \wit M \wit Q_-) + \wit N \wit Q_- \\
& = \La_-^d (\si_+ \Om_-^d + L \Om_-^d \wit P_+)
\La_+^\de(\tau_- \Om_+^\de + M \Om_+^\de \wit Q_-)
+ \La_-^d N \Om_+^\de \wit Q_- \\
& = \La_-^d (\si_+ + L P_+ \wit P_+) \Om_-^d 
\La_+^\de (\tau_- + M Q_- \wit Q_-) \Om_+^\de
+ \La_-^d N Q_- \wit Q_- \Om_+^\de \\
& = \La_-^d \Big( (\si_+ + L P_+) \big( (1 - P_+) + P_+ \wit P_+ \big) \Om_-^d \La_+^\de
( \tau_- + M Q_-) + N Q_- \Big) \\ 
& \qqq \cd \big( (1 - Q_-) + Q_- \wit Q_- \big) \Om_+^\de
\end{split}
\]

Next, we notice that the compositions
\[
\arr{ccc}{
\Om_-^d \La_+^\de : X_+ \to X_+ & & \Om_+^\de \La_-^d : X_- \to X_-
}
\]
are of determinant class and have trivial determinants. This is a consequence of Lemma \ref{l:trivnil}. 

An application of Lemma \ref{l:detequ} therefore yields that
\[
\T{det}(\wit \Si) = \T{det}(\Si) \cd \T{det}\big( (1 - P_+) + P_+ \wit P_+ \big) \cd \T{det}\big( (1 - Q_-) + Q_- \wit Q_- \big)
\]
It is thus enough to show that
\[
\T{det}\big( (1 - P_+) + P_+ \wit P_+ \big)  = 1 = \T{det}\big( (1 - Q_-) + Q_- \wit Q_- \big)
\]

But this follows from the identities $P_+ \wit P_+ P_+ = P_+$ and $Q_- \wit Q_- Q_- = Q_-$, which can be verified by a direct computation.
\end{proof}

A combination of the results of Lemma \ref{l:kapequaleq} and Lemma \ref{l:detequaleq} now shows that the perturbation isomorphism $P(\De,D) : |D| \to |\De|$ is independent of the choice of pseudo-inverses $D^\da$ and $\De^\da$. Since we already know from the work of R. Carey and J. Pincus that the pseudo-inverse $P(\De,D)$ is independent of the choice of perturbation triple $(L,M,N)$ we obtain the following theorem:

\begin{thm}\label{t:indperiso}
The perturbation isomorphism $P(\De,D) : |D| \to |\De|$ associated with the finite rank perturbation of Fredholm complexes $D \to \De$ is independent of the choice of pseudo-inverses $D^\da$ and $\De^\da$ and of the choice of perturbation triple $(L,M,N)$.
\end{thm}

\subsection{Algebraic properties of perturbation isomorphisms}
\emph{Throughout this subsection $D^1 = (X,d^1)$, $D^2 = (X,d^2)$ and $D^3 = (X,d^3)$ will be $\zz/2\zz$-graded Fredholm complexes such that the differences $d^1_+ - d_+^2 \, \, , \, \, d_+^2 - d_+^3 : X_+ \to X_-$ and $d^1_- - d^2_- \, \, , \, \, d_-^2 - d_-^3 : X_- \to X_+$ are of finite rank.} 

We will then state the main algebraic properties of the perturbation isomorphisms. This is the analogue of Theorem \ref{t:algprofre} in the context of Fredholm complexes.

\begin{thm}\label{t:algprocom}
The perturbation isomorphisms satisfy the relations of symmetry and transitivity:
\begin{enumerate}
\item $P(D^1,D^2) = P(D^2,D^1)^{-1}$
\item $P(D^3,D^2) \ci P(D^2,D^1) = P(D^3,D^1)$
\end{enumerate}
\end{thm}
\begin{proof}
This is a straightforward consequence of Theorem \ref{t:algprofre}. Indeed, the perturbation isomorphisms for Fredholm complexes are defined in terms of the perturbation isomorphisms for Fredholm operators, see Definition \ref{d:perisocom}.
\end{proof}

\section{Vanishing of perturbation determinants}\label{s:speper}
\emph{Throughout this section, $D := (X,d)$ will be a $\zz/2\zz$-graded Fredholm complex.}

Choose $n_+, n_- \in \nn_0$ such that $\T{Ind}(D) = n_- - n_+$ and let $C := \big( \ff^{n_+} \op \ff^{n_-}, 0 \big)$. The direct sum $\wit D := D \op C := (\wit X,\wit d)$ is then again a $\zz/2\zz$-graded Fredholm complex and $\T{Ind}(\wit D) = \T{Ind}(D) + n_+ - n_- = 0$.

Fix two linear maps $F_- : \ff^{n_-} \to X_+$ and $F_+ : \ff^{n_+} \to X_-$ such that
\[
d_+ \ci F_- = 0 \q \T{and} \q d_- \ci F_+ = 0
\]
Furthermore, let $N_+ : \ff^{n_+} \to \ff^{n_-}$ be a linear map such that
\[
F_- N_+ = 0 \q \T{and} \q F_+(\T{Ker}(N_+)) = \T{Im}(F_+) 
\]

Define the perturbed differential $\wih d$ on $\wit X$ by
\[
\begin{split}
& \wih d_+ := \ma{cc}{d_+ & F_+ \\ 0 & N_+} : X_+ \op \ff^{n_+} \to X_- \op \ff^{n_-}  \q \T{and} \\
& \wih d_- := \ma{cc}{d_- & F_- \\ 0 & 0} : X_- \op \ff^{n_-}  \to X_+ \op \ff^{n_+}
\end{split}
\]
Let $\wih D := (\wit X,\wih d)$ denote the associated $\zz/2\zz$-graded Fredholm complex.

The main aim of this section is to provide a simple description of the perturbation isomorphism
\[
P(\wih D, \wit D) : |\wit D| \to |\wih D|
\]
associated with the finite rank perturbation $\wit D \to \wih D$. We start with a preliminary lemma.

\begin{lemma}\label{l:injhom}
Let $Q_+ \su \T{Ker}(d_+)$ and $Q_- \su \T{Ker}(d_-)$ be subspaces such that the quotient maps
\[
\begin{split}
& Q_+ \to \T{Ker}(d_+)/\big( \T{Im}(d_-) + \T{Im}(F_-) \big) \q \T{and} \\
& Q_- \to \T{Ker}(d_-)/\big( \T{Im}(d_+)  + \T{Im}(F_+) \big)
\end{split}
\]
are isomorphisms. Let $E_+ : X_+ \to X_+$ and $E_- : X_- \to X_-$ be idempotents with
\[
\begin{split}
& E_+|_{\T{Im}(d_-) + \T{Im}(F_-)} = 0 \, \, , \, \, E_-|_{\T{Im}(d_+) + \T{Im}(F_+)} = 0 \q \T{and} \\
& \T{Im}(E_+) = Q_+ \, \, , \T{Im}(E_-) = Q_-
\end{split}
\]
Furthermore, let $\Om_- : \ff^{n_-} \to \ff^{n_-}$ be an idempotent with $\T{Im}(\Om_-) = \T{Im}(N_+)$.

Then the linear maps
\[
\begin{split}
& \phi_+ := \ma{cc}{E_+ & 0 \\ 0 & 1} : X_+ \op \ff^{n_+} \to X_+ \op \ff^{n_+} \q \T{and} \\
& \phi_- := \ma{cc}{E_- & 0 \\ 0 & 1 - \Om_-} : X_- \op \ff^{n_-} \to X_- \op \ff^{n_-}
\end{split}
\]
induce injective maps $\phi_+ : H_+(\wih D) \to H_+(\wit D)$ and $\phi_- : H_-(\wih D) \to H_-(\wit D)$.
\end{lemma}
\begin{proof}
It is clear that $\phi_+$ induces a map between the homology groups $H_+(\wih D)$ and $H_+(\wit D)$. Suppose thus that $[E_+\xi,\la] = 0$ in $H_+(\wit D)$ for some $(\xi,\la) \in X_+ \op \ff^{n_+}$ with $(d_+ \xi + F_+ \la, N_+ \la) = 0$. It then follows that $\la = 0$ and $E_+ \xi = 0$. We therefore also have that $\xi \in \T{Ker}(d_+)$. Now, since the quotient map $Q_+ \to \T{Ker}(d_+)/\big( \T{Im}(d_-) + \T{Im}(F_-) \big)$ is an isomorphism we may find a $(\eta,\mu) \in X_- \op \ff^{n_-}$ such that $d_- \eta + F_- \mu + E_+ \xi = \xi$. But this shows that $(\xi,\la) = (d_- \eta + F_- \mu,0) \in \T{Im}(\wih d_-)$ and we conclude that $\phi_+ : H_+(\wih D) \to H_+(\wit D)$ is injective.

A similar argument proves that $\phi_-$ induces an injective map $\phi_- : H_-(\wih D) \to H_-(\wih D)$.  
\end{proof}

Define the subspaces 
\[
\begin{split}
W_+ & := \big\{ \la \in \T{Ker}(N_+) \, | \, F_+(\la) \in \T{Im}(d_+)\big\}
\q \T{and} \\
W_- & := \big\{ \la \in \ff^{n_-} \, | \, F_-(\la) \in \T{Im}(d_-) \big\}
\end{split}
\]

We are now ready to state the main result of this section.

\begin{prop}\label{p:pereaz}
Let $L : H_+(\wit D) \to H_-(\wit D)$ and $M : H_-(\wih D) \to H_+(\wih D)$ be isomorphisms. Suppose that there exist subspaces $V_+ \su \T{Ker}(N_+)$, $Q_+ \su \T{Ker}(d_+)$, $Z_+ \su \T{Ker}(F_+)$ and $V_- \su \ff^{n_-}$, $Q_- \su \T{Ker}(d_-)$ such that
\begin{enumerate}
\item $\T{Ker}(N_+) = V_+ \dop W_+$, $\ff^{n_+} = V_+ \dop W_+ \dop Z_+$ and $\ff^{n_-} = V_- \dop W_-$
\item The quotient maps $Q_+ \to \T{Ker}(d_+)/\big( \T{Im}(d_-) + \T{Im}(F_-) \big)$ and $Q_- \to \T{Ker}(d_-)/\big( \T{Im}(d_+) + \T{Im}(F_+)\big)$ are isomorphisms.
\item $L[F_-(v_-), v_+ + z_+] = [F_+(v_+),v_- + N_+ z_+]$ for all $v_+ \in V_+$, $z_+ \in Z_+$ and $v_- \in V_-$.
\item The diagram
\[
\begin{CD}
H_-(\wih D) @>\phi_->> H_-(\wit D) \\
@V{M}VV @VV{L^{-1}}V \\
H_+(\wih D) @>\phi_+>> H_+(\wit D)
\end{CD}
\]
commutes, where $\phi_+$ and $\phi_-$ are defined as in Lemma \ref{l:injhom}.
\end{enumerate}
The perturbation isomorphism $P(\wih D,\wit D)$ is then given by
\[
P(\wih D, \wit D) : s_+ \ot (Ls_+)^* \mapsto Ms_- \ot s_-^*
\]
for all non-zero vectors $s_+ \in \T{det}\big( H_+(\wit D) \big)$ and $s_- \in \T{det}\big( H_-(\wih D) \big)$.
\end{prop}

The proof of the above proposition will be carried out in several steps. The first step consists of constructing reasonable pseudo-inverses of $\wit D$ and $\wih D$.

To this end, choose a pseudo-inverse 
\[
\begin{CD}
D^\da : X_- @>{d_+^\da}>> X_+ @>{d_-^\da}>> X_-
\end{CD}
\]
of $D$ such that 
\[
\begin{split}
& (d_-^\da F_-)|_{V_-} = 0 \, \, , \,\, d_-^\da E_+ = 0 \, \, , \, \, E_- d_-^\da = 0 
\q \T{and} \\ 
& (d_+^\da F_+)|_{V_+} = 0 \, \, , \, \, d_+^\da E_- = 0 \, \, , \, \, E_+ d_+^\da = 0
\end{split}
\]

The $\zz/2\zz$-graded Fredholm complex
\[
\begin{CD}
\wit D^\da : X_- \op \ff^{n_-} @>\ma{cc}{d_+^\da & 0 \\ 0 & 0}>> X_+ \op \ff^{n_+} @>\ma{cc}{d_-^\da & 0 \\ 0 & 0}>> X_- \op \ff^{n_-}
\end{CD}
\]
is then a pseudo-inverse of $\wit D$.

In order to find a pseudo-inverse of $\wih D$, define the maps $G_- : X_- \to \ff^{n_+}$ and $G_+ : X_+ \to \ff^{n_-}$ by
\[
\begin{split}
& G_-|_{Q_- + \T{Im}(d_+) + \T{Im}(d_-^\da)} = 0 \, \, , \, \, G_-(F_+(v_+)) = v_+ \q \T{and} \\
& G_+|_{Q_+ + \T{Im}(d_-) + \T{Im}(d_+^\da)} = 0 \, \, , \, \, G_+(F_-(v_-)) = v_-
\end{split}
\]
for all $v_+ \in V_+$ and $v_- \in V_-$. Furthermore, define a pseudo-inverse of $N_+ : \ff^{n_+} \to \ff^{n_-}$ such that $V_- \dop (1 - \Om_-) W_- = \T{Ker}(N_+^\da)$ and $\T{Im}(N_+^\da) = Z_+$. Here we are using the fact that $\ff^{n_-} = V_- \dop (1 - \Om_-) W_- \dop \T{Im}(N_+)$.

\begin{lemma}
The $\zz/2\zz$-graded Fredholm complex
\[
\begin{CD}
\wih{D}^\da : X_- \op \ff^{n_-} @>\ma{cc}{d_+^\da & 0 \\ G_- & N_+^\da}>> X_+ \op \ff^{n_+} @>\ma{cc}{d_-^\da & 0 \\ G_+ & 0}>> X_- \op \ff^{n_-}
\end{CD}
\]
is a pseudo-inverse of $\wih D$.
\end{lemma}
\begin{proof}
Notice first that $\wih{d_-}^\da \wih{d_+}^\da = 0$ and $\wih{d_+}^\da \wih{d_-}^\da = 0$. It is therefore enough to prove the identities
\[
\wih d_+ \wih{d_+}^\da \wih d_+ = \wih d_+ \, \, , \, \, \wih{d_+}^\da \wih d_+ \wih{d_+}^\da = \wih{d_+}^\da \q \T{and} \q
\wih d_- \wih{d_-}^\da \wih d_- = \wih d_- \, \, , \, \, \wih{d_-}^\da \wih d_- \wih{d_-}^\da = \wih{d_-}^\da
\]

We will focus on proving the two identities involving $\wih d_+$ and $\wih{d_+}^\da$ since the proof of the remaining two identities is similar.

To this end, remark that
\[
\wih d_+ \wih{d_+}^\da = \ma{cc}{d_+ & F_+ \\ 0 & N_+} \ma{cc}{d_+^\da & 0 \\ G_- & N_+^\da} = \ma{cc}{F_+ G_- + d_+ d_+^\da  & 0 \\ 0 & N_+ N_+^\da}
\]
This implies that
\[
\wih d_+ \wih{d_+}^\da \wih d_+ = \ma{cc}{d_+ & F_+ G_- F_+ + d_+ d_+^\da F_+ \\ 0 & N_+} = \wih d_+
\]
and furthermore that
\[
\wih{d_+}^\da \wih d_+ \wih{d_+}^\da = \ma{cc}{d_+^\da & 0 \\ G_- F_+ G_- & N_+^\da} = \wih{d_+}^\da
\]

These computations prove the lemma.
\end{proof}

We can now define the Fredholm operators
\[
\begin{split}
& \si_+ := \wit{d_+} + \wit{d_-}^\da \, \, \, , \, \, \, \tau_+ := \wih{d_+} + \wih{d_-}^\da 
: X_+ \op \ff^{n_+} \to X_- \op \ff^{n_-} \q \T{and} \\
& \si_- := \wit{d_-} + \wit{d_+}^\da \, \, \, , \, \, \, \tau_- := \wih{d_-} + \wih{d_+}^\da 
: X_- \op \ff^{n_-} \to X_+ \op \ff^{n_+}
\end{split}
\]

The next step in the proof of Proposition \ref{p:pereaz} is to construct appropriate isomorphisms
\[
\sL : \T{Ker}(\si_+) \to \T{Ker}(\si_-) \q \T{and} \q \sM : \T{Ker}(\tau_-) \to \T{Ker}(\tau_+)
\]
which induce $\T{det}(L) : \T{det}\big( H_+(\wit D) \big) \to \T{det}\big( H_-(\wit D) \big)$ and $\T{det}(M) : \T{det}\big( H_-(\wih D)\big) \to \T{det}\big( H_+(\wih D) \big)$ at the level of determinants of homology groups.

To carry out this step, we need a better understanding of the homology groups $H_+(\wih D)$ and $H_-(\wih D)$. This is provided by the following lemma:

\begin{lemma}\label{l:proiso}
The idempotents $\Pi_+ := 1 - \tau_- \tau_+$ and $\Pi_- := 1 - \tau_+ \tau_-$ induce isomorphisms
\[
\Pi_+ : Q_+ \op W_+ \to \T{Ker}(\tau_+) \q \T{and} \q \Pi_- : Q_- \op (1- \Om_-) W_- \to \T{Ker}(\tau_-)
\]
The inverses are induced by the maps $\phi_+ \in \sL(\wit X_+)$ and $\phi_- \in \sL(\wit X_-)$.
\end{lemma}
\begin{proof}
We will only prove the claim on $\Pi_+$ since the case of $\Pi_-$ follows a similar pattern.

Remark first that $\phi_+(\xi,\la) = (E_+ \xi,\la) \in Q_+ \op W_+$ whenever $(\xi,\la) \in \T{Ker}(\tau_+)$. Indeed, since $E_+ \xi \in Q_+$ by definition of $E_+$ it suffices to show that $\la \in W_+$. But this follows by noting that $N_+ \la = 0$ and $F_+ \la \in \T{Im}(d_+)$ since $F_+ \la + d_+ \xi = 0$.

To continue, we compute that 
\begin{equation}\label{eq:compip}
\begin{split}
\Pi_+ & = 1 - \wih{d_+}^\da \wih{d_+} - \wih{d_-} \wih{d_-}^\da \\
& = 1 - \ma{cc}{d_+^\da & 0 \\ G_- & N_+^\da} \ma{cc}{d_+ & F_+ \\ 0 & N_+} - \ma{cc}{d_- & F_- \\ 0 & 0} \ma{cc}{d_-^\da & 0 \\ G_+ & 0} \\
& = \ma{cc}{ 1 - d_+^\da d_+ - d_- d_-^\da - F_- G_+ & - d_+^\da F_+ \\ 0 & 1 - G_- F_+ - N_+^\da N_+}
\end{split}
\end{equation}

The above computation implies that
\[
\phi_+ \Pi_+ = \ma{cc}{E_+ & 0 \\ 0 & 1 - G_- F_+ - N_+^\da N_+}
\]
This shows that $(\phi_+ \Pi_+)|_{Q_+ \op W_+} = 1$. To prove the claim on $\Pi_+$ it is therefore enough to show that $\phi_+ : \T{Ker}(\tau_+) \to Q_+ \op W_+$ is injective. But this is a consequence of Lemma \ref{l:injhom}.
\end{proof}

Let us now choose an isomorphism $\sL : \T{Ker}(\si_+) \to \T{Ker}(\si_-)$ such that 
\[
\begin{split}
\sL(F_-(v_-), v_+ + z_+) & = (F_+(v_+), v_- + N_+ z_+) \q \T{and} \\ 
\sL(Q_+ \op W_+) & = Q_- \op (1 - \Om_-) W_- 
\end{split}
\]
for all $v_- \in V_-$, $v_+ \in V_+$ and $z_+ \in Z_+$. We may furthermore arrange that $\sL$ induces $\T{det}(L) : \T{det}\big( H_+(\wit D) \big) \to \T{det}\big( H_-(\wit D) \big)$ at the level of determinants of homology groups. Remark here that $\T{Ker}(\si_+) = \big( F_- V_- \dop Q_+\big) \op \ff^{n_+}$ and $\T{Ker}(\si_-) = \big( Q_- \dop F_+(V_+)  \big) \op \ff^{n_-}$.

Choose the isomorphism $\sM := \Pi_+ \sL^{-1} \phi_- : \T{Ker}(\tau_-) \to \T{Ker}(\tau_+)$, see Lemma \ref{l:proiso}. It then follows from condition (4) of Proposition \ref{p:pereaz} that $\sM$ induces $\T{det}(M) : \T{det}\big( H_-(\wih D) \big) \to \T{det}\big( H_+(\wih D) \big)$ at the level of determinants of homology groups.

To prove Proposition \ref{p:pereaz} it therefore suffices to show that the determinant of $\Si := (\tau_- + \sM \Pi_-)(\si_+ + \sL P_+) : \wit X_+ \to \wit X_+$ is equal to one, where $P_+ := 1 - \si_- \si_+$.

\begin{lemma}
\[
\T{det}(\Si) = 1
\]
\end{lemma}
\begin{proof}
We compute each of the four terms in the product $\Si = \tau_- \si_+ + \sM \Pi_- \si_+ + \tau_- \sL P_+ + \sM \Pi_- \sL P_+$ separately.

The first term is given by
\[
\tau_- \si_+ = \ma{cc}{(d_- + d_+^\da)(d_+ + d_-^\da) & 0 \\ G_- (d_+ + d_-^\da) & 0} = 1 - P_+
\]

To compute the second term we note that a computation similar to \eqref{eq:compip} yields that
\[
\Pi_- = \ma{cc}{1 - d_-^\da d_- - d_+ d_+^\da - F_+ G_- & - d_-^\da F_- \\ 0 & 1 - G_+ F_- - N_+ N_+^\da}
\]
This implies that
\[
\Pi_- \si_+ = \ma{cc}{- F_+ G_- (d_+ + d_-^\da) & 0 \\ 0 & 0} = 0
\]
In particular we have that $\sM \Pi_- \si_+ = 0$.

In order to compute the third term, let $v_+ \in V_+$, $v_- \in V_-$, $w_+ \in W_+$, $z_+ \in Z_+$ and $q_+ \in Q_+$. Then
\[
\begin{split}
(\tau_- \sL)(F_-(v_-) + q_+, v_+ + w_+ + z_+) 
& = \tau_- \sL(q_+,w_+) + \tau_-(F_+ v_+, v_- + N_+ z_+) \\
& = \ma{cc}{0 & F_- \\ 0 & 0} \sL(q_+,w_+) + (F_- v_-, v_+ + z_+)
\end{split}
\]

The last term is given by
\[
\begin{split}
& (\sM \Pi_- \sL)(F_-(v_-) + q_+, v_+ + w_+ + z_+) 
= (\Pi_+ \sL^{-1} \phi_- \Pi_- \sL)(q_+,w_+) \\ 
& \qq = \Pi_+(q_+,w_+) = (q_+ - d_+^\da F_+ w_+, w_+)
\end{split}
\]
for all $v_+ \in V_+$, $v_- \in V_-$, $w_+ \in W_+$, $z_+ \in Z_+$ and $q_+ \in Q_+$.

It follows from the above computations that
\[
\begin{split}
\Si & = 1 + \ma{cc}{0 & F_- \\ 0 & 0} \sL \ma{cc}{E_+ & 0 \\ 0 & 1 - G_- F_+ - N_+^\da N_+} \\
& \qq + \ma{cc}{0 & - d_+^\da F_+ (1 - G_- F_+ - N_+^\da N_+) \\ 0 & 0}
\end{split}
\]
But this shows that $\T{det}(\Si) = 1$ since the finite rank operator 
\[
\begin{split}
& \ma{cc}{0 & F_- \\ 0 & 0} \sL \ma{cc}{E_+ & 0 \\ 0 & 1 - G_- F_+ - N_+^\da N_+} \\
& \qq + \ma{cc}{0 & - d_+^\da F_+ (1 - G_- F_+ - N_+^\da N_+) \\ 0 & 0}
: \wit X_+ \to \wit X_+
\end{split}
\]
has trivial square.
\end{proof}

\section{Perturbations of mapping cone triangles}\label{s:permaptri}
\emph{Throughout this section, $D^1 := (X^1,d^1)$ and $D^2 := (X^2,d^2)$ will be $\zz/2\zz$-graded Fredholm complexes. Furthermore, these two complexes will be linked by a chain map $A : D^1 \to D^2$.} The mapping cone of $A$ will be denoted by $D := (X,d) := C^A$.

The notation $T(\G D) : |D^2| \to |D^1| \ot |D|$ will refer to the torsion isomorphism associated with the mapping cone triangle
\[
\begin{CD}
D^1 @>{A}>> D^2 @>{i}>> D @>{p}>> TD^1
\end{CD}
\]
See Definition \ref{d:torisocha}.

On top of this data, we will consider alternative differentials $\de^1 : X^1 \to X^1$ and $\de^2 : X^2 \to X^2$ on the chains of $D^1$ and $D^2$, respectively. The resulting complexes are denoted by $\De^1 := (X^1,\de^1)$ and $\De^2 := (X^2,\de^2)$. Let also $B : \De^1 \to \De^2$ be a chain map and denote the mapping cone complex by $\De := (X,\de) := C^B$.
\medskip

\emph{It will be a standing assumption that $d_+ - \de_+ : X_+ \to X_-$ and $d_- - \de_- : X_- \to X_+$ have finite rank.}
\medskip

It follows from this assumption that the complexes $\De^1$, $\De^2$ and $\De$ are Fredholm. Furthermore their determinants are linked to the determinants of $D^1$, $D^2$ and $D$ by the perturbation isomorphisms,
\[
P(\De^1,D^1) : |D^1| \to |\De^1| \q P(\De^2,D^2) : |D^2| \to |\De^2| \, \, \, \T{and} \, \, \, P(\De,D) : |D| \to |\De|
\]
See Subsection \ref{ss:perfrecom}.

We will apply the notation $T(\G \De) : |\De^2| \to |\De^1| \ot |\De|$ for the torsion isomorphism associated with the mapping cone triangle
\[
\begin{CD}
\De^1 @>{B}>> \De^2 @>{j}>> \De @>{q}>> T\De^1
\end{CD}
\]

The main result of this section can now be announced. It provides a fundamental relation between perturbation isomorphisms and torsion isomorphisms.

\begin{thm}\label{t:pertor}
Suppose that $d_+ - \de_+ : X_+ \to X_-$ and $d_- - \de_- : X_- \to X_+$ have finite rank and that $D^1$ and $D^2$ are Fredholm. Then the following diagram commutes,
\[
\begin{CD}
|D^2| @>{T(\G D)}>> |D^1|\ot |D| \\
@V{P(\De^2,D^2)}VV @VV{P(\De^1,D^1) \ot P(\De,D)}V \\
|\De^2| @>>{T(\G \De)}> |\De^1| \ot |\De|
\end{CD}
\]
\end{thm}

The proof of this theorem will be carried out in several steps during the next subsections. In the first subsection we shall see how we can reduce the theorem to the case where all the involved complexes have index zero. In the second subsection we will reduce the problem further to the case where the subcomplexes $D^2$ and $\De^2$ are exact. In the third subsection we will prove the theorem when $D^2$ and $\De^2$ are assumed to be exact and then finally, in the last subsection, we shall recollect our results and give a full proof of Theorem \ref{t:pertor}.

\subsection{Reduction to index zero}\label{ss:redzer}
Let $n_+, n_- \in \nn_0$ and $m_+, m_- \in \nn_0$ be non-negative integers with $\T{Ind}(D^1) = n_- - n_+$ and $\T{Ind}(D^2) = m_- - m_+$.

Consider the $\zz/2\zz$-graded Fredholm complexes $C^1 := \big( \ff^{n_+} \op \ff^{n_-},0 \big)$ and $C^2 := \big( \ff^{m_+} \op \ff^{m_-},0 \big)$. Form the direct sums $\wit D^1 := D^1 \op C^1$, $\wit D^2 := D^2 \op C^2$ and $\wit D := D \op TC^1 \op C^2$. Remark that all of these $\zz/2\zz$-graded Fredholm complexes have index zero.

Let $T(\wit{\G D}) : |\wit D^2| \to |\wit D^1| \ot |\wit D|$ denote the torsion isomorphism associated with the six term exact sequence
\[
\begin{CD}
H_+(D^1) \op \ff^{n_+} @>{(A_+ \op 0)}>> H_+(D^2) \op \ff^{m_+} @>{i_+ \op \io_+}>> H_+(D) \op \ff^{n_-} \op \ff^{m_+} \\
@A{p_- \op \pi_-}AA & & @VV{p_+ \op \pi_+}V \\
H_-(D) \op \ff^{n_+} \op \ff^{m_-} @<<{i_- \op \io_-}< H_-(D^2) \op \ff^{m_-} @<<{(A_- \op 0)}< H_-(D^1) \op \ff^{n_-}
\end{CD}
\]
where $\io_+ : \ff^{m_+} \to \ff^{n_-} \op \ff^{m_+}$, $\io_- : \ff^{m_-} \to \ff^{n_+} \op \ff^{m_-}$ and $\pi_+ : \ff^{n_-} \op \ff^{m_+} \to \ff^{n_-}$, $\pi_- : \ff^{n_+} \op \ff^{m_-} \to \ff^{n_+}$ are the obvious inclusion and projection maps.

We may apply a similar construction to the $\zz/2\zz$-graded Fredholm complexes $\De^1$, $\De^2$ and $\De$. In this way we obtain a torsion isomorphism $T(\wit{\G \De}) : |\wit \De^2| \to |\wit \De^1| \ot |\wit \De|$. 

We denote the perturbation isomorphisms coming from the finite rank perturbations $\wit D^1 \to \wit \De^1$, $\wit D^2 \to \wit \De^2$ and $\wit D \to \wit \De$ by
\[
P\big( \wit \De^1, \wit D^1\big) : |\wit D^1| \to |\wit \De^1| \, \, , \, \,
P\big( \wit \De^2, \wit D^2 \big) : |\wit D^2| \to |\wit \De^2| \, \, \T{ and } \, \,
P\big( \wit \De, \wit D \big) : |\wit D| \to |\wit \De|
\]

The main result of this subsection can now be announced. It provides the first step in the proof of Theorem \ref{t:pertor}. Indeed, with the next proposition in hand, we only need to prove Theorem \ref{t:pertor} in the case where all the involved complexes have index zero.

\begin{prop}\label{p:equdia}
The diagram
\begin{equation}\label{eq:diaspeI}
\begin{CD}
|D^2| @>{T(\G D)}>> |D^1| \ot |D| \\
@V{P(\De^2,D^2)}VV @VV{P(\De^1,D^1) \ot P(\De,D)}V \\
|\De^2| @>>{T(\G \De)}> |\De^1| \ot |\De|
\end{CD}
\end{equation}
commutes if and only if the diagram
\begin{equation}\label{eq:diaspeII}
\begin{CD}
|\wit D^2| @>{T(\wit{\G D})}>> |\wit D^1| \ot |\wit D| \\
@V{P(\wit \De^2,\wit D^2)}VV @VV{P(\wit \De^1,\wit D^1) \ot P(\wit \De,\wit D)}V \\
|\wit \De^2| @>>{T(\wit{\G \De})}> |\wit \De^1| \ot |\wit \De|
\end{CD}
\end{equation}
commutes.
\end{prop}

The proof of Proposition \ref{p:equdia} will rely on two lemmas.

Let us fix non-zero vectors $\om_+^1 \in |\ff^{n_+}|$, $\om_-^1 \in |\ff^{n_-}|$ and $\om_+^2 \in |\ff^{m_+}|$, $\om_-^2 \in |\ff^{m_-}|$. We then have the isomorphisms
\[
\begin{split}
& i_{\om^1} : |D^1| \to |\wit D^1| \q i_{\om^1} : s_+^1 \ot (s_-^1)^* \mapsto (s_+^1 \we \om_+^1) \ot (s_-^1 \we \om_-^1)^* \\
& i_{\om^2} : |D^2| \to |\wit D^2| \q i_{\om^2} : s_+^2 \ot (s_-^2)^* \mapsto (s_+^2 \we \om_+^2) \ot (s_-^2 \we \om_-^2)^* \\
& i_{\om} : |D| \to |\wit D| \q i_{\om} : s_+ \ot s_-^* \mapsto (s_+ \we \om_-^1 \we \om_+^2) \ot (s_- \we \om_+^1 \we \om_-^2)^*
\end{split}
\]
The same notation will be applied for the corresponding isomorphisms for the $\zz/2\zz$-graded Fredholm complexes $\De^1$, $\wit{\De}^1$, etc.

The first lemma provides a relation between the torsion isomorphisms $T(\G D) : |D^2| \to |D^1| \ot |D|$ and $T(\wit{\G D}) : |\wit D^2| \to |\wit D^1| \ot |\wit D|$. We notice that a similar relation will hold for the torsion isomorphisms $T(\G \De) : |\De^2| \to |\De^1| \ot |\De|$ and $T(\wit{\G \De}) : |\wit \De^2| \to |\wit \De^1| \ot |\wit \De|$.

\begin{lemma}\label{l:dilzer}
The diagram
\[
\begin{CD}
|D^2| @>{T(\G D)}>> |D^1| \ot |D| \\
@V{i_{\om^2}}VV @VV{i_{\om^1} \ot i_{\om}}V \\
|\wit D^2| @>>{T(\wit{\G D})}> |\wit D^1| \ot |\wit D|
\end{CD}
\]
commutes up to the sign $(-1)^{n_- + n_+ \cd (m_+ + m_-)}$.
\end{lemma}
\begin{proof}
Fix non-zero vectors
\[
\begin{split}
& t^1_+ \in |H_+(D^1)_{(1)}| \, \, , \, \, t^2_+ \in |H_+(D^2)_{(1)}| \, \, , \, \, t_+ \in |H_+(D)_{(1)}| \q \T{and} \\
& t^1_- \in |H_-(D^1)_{(1)}| \, \, , \, \, t^2_- \in |H_-(D^2)_{(1)}| \, \, , \, \, t_- \in |H_-(D)_{(1)}|
\end{split}
\]
Applying the definition of the maps involved we obtain that
\[
\begin{split}
& \big( (i_{\om^1} \ot i_{\om}) \ci T(\G D) \big)\Big( A_+ t_+^1 \we t_+^2 \ot \big( A_- t_-^1 \we t_-^2\big)^* \Big) \\
& \q = p_- t_- \we t_+^1 \we \om_+^1 \ot (p_+ t_+ \we t_-^1 \we \om_-^1)^* \\ 
& \qqq \ot i_+ t_+^2 \we t_+ \we \om_-^1 \we \om_+^2 \ot (i_- t_-^2 \we t_- \we \om_+^1 \we \om_-^2)^* \cd (-1)^{\mu(\G D)}  
\end{split}
\]
where the sign is given by
\[
\begin{split}
\mu(\G D) & = (\ep(t_+^2) + 1)\cd \big(\ep(t_-^1) + \ep(t_+^1)\big) + \ep(t_-^1) \cd \big(\ep(t_+) + \ep(t_-)\big) \\
& \q + \ep(t_-) \cd \big(\ep(t_+^2) + \ep(t_-^2)\big) + \ep(t_+)
\end{split}
\]
See Definition \ref{d:toriso}. On the other hand we have that
\[
\begin{split}
& \big(T(\wit{\G D}) \ci i_{\om^2}\big)\big( A_+ t_+^1 \we t_+^2 \ot ( A_- t_-^1 \we t_-^2)^* \big) \\
& \q = T(\wit{\G D})\big( A_+ t_+^1 \we t_+^2 \we \om_+^2 \ot ( A_- t_-^1 \we t_-^2 \we \om_-^2)^* \\
& \q = p_- t_- \we \om_+^1 \we t_+^1 \ot (p_+ t_+ \we \om_-^1 \we t_-^1)^* \\ 
& \qqq \ot
i_+ t_+^2 \we \om_+^2 \we t_+ \we \om_-^1 \ot (i_- t_-^2 \we \om_-^2 \we t_- \we \om_+^1)^* \cd (-1)^{\mu(\wit{\G D})} 
\end{split}
\]
where the sign is given by
\[
\begin{split}
\mu(\wit{\G D}) & = \big(\ep(t_+^2) + m_+ + 1\big) \cd \big( \ep(t_-^1) + \ep(t_+^1) \big) + \ep(t_-^1) \cd \big( \ep(t_+) + \ep(t_-) + n_+ + n_-\big) \\ 
& \q + \big( \ep(t_-) + n_+ \big) \cd \big( \ep(t_+^2) + m_+ + \ep(t_-^2) + m_- \big) + \ep(t_+) + n_- 
\end{split}
\]

It is therefore enough to show that
\[
(-1)^{\mu(\G D) + \mu(\wit{\G D}) + n_+ \cd \ep(t_+^1) + n_- \cd \ep(t_-^1) + m_+ \cd \big( \ep(t_+) + n_- \big) + m_- \cd \big(\ep(t_-) + n_+ \big)}
= (-1)^{n_- + n_+ \cd (m_+ + m_-)}
\]
But this identity is verified by the following straightforward computation,
\[
\begin{split}
& (-1)^{\mu(\G D) + \mu(\wit{\G D}) + n_+ \cd \ep(t_+^1) + n_- \cd \ep(t_-^1) + m_+ \cd \big( \ep(t_+) + n_- \big) + m_- \cd \big(\ep(t_-) + n_+ \big)} \\
& \q = (-1)^{n_+ \cd \big( \ep(t_+^1) + \ep(t_+^2) + \ep(t_-^1) + \ep(t_-^2) \big) + m_+ \cd \big( \ep(t_-^1) + \ep(t_-) + \ep(t_+^1) + \ep(t_+)\big)  + m_+ \cd \big( n_- + n_+ \big) + n_-} \\
& \q = (-1)^{n_+ \cd \big( m_+ + m_-\big) + n_-}
\end{split}
\]
where we have applied that $\T{Ind}(D^1) = n_- - n_+$ and $\T{Ind}(D^2) = m_- - m_+$.
\end{proof}

It follows from the above lemma that the diagram in \eqref{eq:diaspeII} commutes if and only if the diagram
\[
\begin{CD}
|D^2| @>{T(\G D)}>> |D^1| \ot |D| \\
@V{i_{\om^2}^{-1} \ci P(\wit{\De}^2,\wit{D}^2) \ci i_{\om^2}}VV @VV{(i_{\om^1}^{-1} \ci P(\wit{\De}^1,\wit{D}^1) \ci i_{\om^1}) \ot (i_{\om}^{-1} \ci P(\wit{\De},\wit{D}) \ci i_{\om}) }V \\
|\De^2| @>>{T(\G \De)}> |\De^1| \ot |\De|
\end{CD}
\]
commutes. The result of Proposition \ref{p:equdia} will therefore be a consequence of the next lemma.

\begin{lemma}\label{l:dirsumper}
\[
P(\De^1,D^1) = i_{\om^1}^{-1} \ci P(\wit{\De}^1,\wit{D}^1) \ci i_{\om^1}
\]
\end{lemma}
\begin{proof}
We will only consider the case where $n_+ \geq n_-$, thus where $\T{Ind}(D^1) \leq 0$, since the proof in the case where $n_- \geq n_+$ follows a similar pattern.

Choose pseudo-inverses
\[
\begin{CD}
(D^1)^\da : X_-^1 @>{(d_+^1)^\da}>> X_+^1 @>{(d_-^1)^\da}>> X_-^1 \q \T{and} \q
(\De^1)^\da : X_-^1 @>{(\de_+^1)^\da}>> X_+^1 @>{(\de_-^1)^\da}>> X_-^1
\end{CD}
\]
of $D^1$ and $\De^1$ such that $(D^1)^\da \to (\De^1)^\da$ is a finite rank perturbation. It follows that the $\zz/2\zz$-graded Fredholm complexes
\[
\begin{CD}
(\wit{D}^1)^\da : X_-^1 \op \ff^{n_-} @>{(d_+^1)^\da \op 0 }>> X_+^1 \op \ff^{n_+} @>{(d_-^1)^\da \op 0}>> X_-^1 \op \ff^{n_-} & & \q \T{and} \\
(\wit{\De}^1)^\da : X_-^1 \op \ff^{n_-} @>{(\de_+^1)^\da \op 0}>> X_+^1 \op \ff^{n_+} @>{(\de_-^1)^\da \op 0}>> X_-^1 \op \ff^{n_-} & &
\end{CD}
\]
are pseudo-inverses of $\wit D^1$ and $\wit \De^1$ such that $(\wit D^1)^\da \to (\wit \De^1)^\da$ is a finite rank perturbation. 

The Fredholm operators associated with $D^1 \to \De^1$ and $(D^1)^\da \to (\De^1)^\da$ are denoted by $\si_+^1, \tau_+^1 : X_+^1 \to X_-^1$ and $\si_-^1, \tau_-^1 : X_-^1 \to X_+^1$. The idempotents are denoted by $Q_+^1, \Pi_+^1 : X_+^1 \to X_+^1$ and $Q_-^1, \Pi_-^1 : X_-^1 \to X_-^1$. The Fredholm operators and idempotents associated with the finite rank perturbations $\wit D^1 \to \wit \De^1$ and $(\wit D^1)^\da \to (\wit \De^1)^\da$ share the same notation except for an extra ``$\wit{\, \, \, \, \, \, }$'' on top.

Let us choose a perturbation triple $(L,M,N)$ for the finite rank perturbation $\si_+^1 \to \tau_+^1$. The associated isomorphism of determinant class is then given by $\Si^1 := (\si_+^1 + L Q_+^1)(\tau_-^1 + M \Pi_-^1) + N \Pi_-^1 : X_-^1 \to X_-^1$.

Furthermore, let us write $\ff^{n_+} = \ff^{n_+ - n_-} \dop \ff^{n_-}$ and let $i_1 : \ff^{n_+ - n_-} \to \ff^{n_+}$, $i_2 : \ff^{n_-} \to \ff^{n_+}$ and $p_1 : \ff^{n_+} \to \ff^{n_+ - n_-}$, $p_2 : \ff^{n_+} \to \ff^{n_-}$ denote the associated inclusions and projections. Since $\T{Ind}(D^1) = n_- - n_+$ we may choose an isomorphism $\al : \ff^{n_+ - n_-} \to \T{Im}(N)$. It then follows that the linear maps
\[
\begin{split}
\wit L & := \ma{cc}{L & \al p_1 \\ 0 & p_2} : \T{Ker}(\si_+^1) \op \ff^{n_+} \to X_-^1 \op \ff^{n_-} \q \T{and} \\
\wit M & := \ma{cc}{M & 0 \\ i_1 \al^{-1} N & i_2} : \T{Ker}(\tau_-^1) \op \ff^{n_-} \to X_+^1 \op \ff^{n_+}
\end{split}
\]
form part of a perturbation triple $(\wit L, \wit M, 0)$ for the finite rank perturbation $\wit{\si_+}^1 \to \wit{\tau_+}^1$. The associated isomorphism of determinant class is then given by
\[
\wit{\Si}^1 := \ma{cc}{\si_+ + L Q_+^1 & \al p_1 \\ 0 & p_2} \cd \ma{cc}{\tau_-^1 + M \Pi_-^1 & 0 \\ i_1 \al^{-1} N \Pi_-^1 & i_2}
= \ma{cc}{\Si^1 & 0 \\ 0 & 1}
\]
It follows in particular that $\T{det}(\wit{\Si}^1) = \T{det}(\Si^1)$.

Notice now that there exists a unique non-trivial vector $\xi_+ \in |\ff^{n_+ - n_-}|$ such that $\om_+^1 = i_1(\xi_+) \we i_2(\om_-^1)$. Let $s_+ \in |H_+(D^1)|$, $r_- \in |\T{Ker}(N)|$ be non-trivial vectors. Furthermore, let $t_- \in |\T{Ker}(M)|$ be the unique vector with $\al(\xi_+) = N(t_-)$. It follows that
\[
i_{\om^1} \ci P(\De^1,D^1) : s_+ \ot (L s_+ \we N t_-)^* \mapsto \T{det}(\Si^1)^{-1} \cd M r_- \we \om_+^1 \ot (r_- \we t_- \we \om_-^1)^*
\]
On the other hand, since $\wit{L}(s_+ \we \om_+^1) = L s_+ \we \al \xi_+ \we \om_-^1 = L s_+ \we N t_- \we \om_-^1$, we have that
\[
P(\wit{\De}^1,\wit{D}^1) \ci i_{\om^1} : s_+ \ot (L s_+ \we N t_-)^* \mapsto \T{det}(\wit{\Si}^1)^{-1} \cd \wit{M}(r_- \we t_- \we \om_-^1) \ot (r_- \we t_- \we \om_-^1)^*
\]
But this proves the lemma since $\T{det}(\wit{\Si}^1)^{-1} = \T{det}(\Si^1)^{-1}$ and $\wit{M}(r_- \we t_- \we \om_-^1) = M r_- \we i_1 \xi_+ \we i_2 \om_-^2 = M r_- \we \om_+^1$.
\end{proof}

\subsection{Index zero}
As in Subsection \ref{ss:redzer}, let $n_+,n_- \in \nn_0$ and $m_+,m_- \in \nn_0$ be non-negative integers with $n_- - n_+ = \T{Ind}(D^1) = \T{Ind}(\De^1)$ and $m_- - m_+ = \T{Ind}(D^2) = \T{Ind}(\De^2)$. Without loss of generality, we may suppose that the inequalities 
\[
\begin{split}
m_+ \geq \T{dim}(H_-(D^2)) \, , \, \T{dim}(H_-(\De^2)) \q , \q m_- \geq \T{dim}(H_+(D^2)) \, , \, \T{dim}(H_+(\De^2))
\end{split}
\]
hold.

Let us choose subspaces $Q^2_+ \su \T{Ker}(d^2_+)$ and $Q^2_- \su \T{Ker}(d^2_-)$ such that
\[
\T{Ker}(d^2_+) = \T{Im}(d^2_-) \dop Q^2_+ \q \T{Ker}(d^2_-) = \T{Im}(d^2_+) \dop Q^2_-
\]
Choose linear maps $F^2_+ : \ff^{m_+} \to X^2_-$ and $F^2_-: \ff^{m_-} \to X^2_+$ such that
\[
\T{Im}(F^2_+) = Q^2_- \q \T{and} \q \T{Im}(F^2_-) = Q^2_+
\]
Furthermore, let $Z_+^2 := \T{Ker}(F^2_+)$, $Z_-^2 := \T{Ker}(F^2_-)$ and choose subspaces $V_+^2 \su \ff^{m_+}$, $V_-^2 \su \ff^{m_-}$ with 
\[
\ff^{m_+} = V_+^2 \dop Z_+^2 \q \T{and} \q \ff^{m_-} = V_-^2 \dop Z_-^2
\]
Finally, since $0 = \T{Ind}(D^2) + m_+ - m_- = \T{dim}(Z_+^2) - \T{dim}(Z_-^2)$, we may choose a linear map $N_+^2 : \ff^{m_+} \to \ff^{m_-}$ such that
\[
\T{Ker}(N_+^2) = V_+^2 \q \T{and} \q \T{Im}(N_+^2) = \T{Ker}(F_-^2)
\]

Consider now the perturbed differentials on the chains of $\wit D^2$,
\[
\begin{split}
& \wih{d_+}^2 := \ma{cc}{d_+^2 & F_+^2 \\ 0 & N_+^2} : X_+^2 \op \ff^{m_+} \to X_-^2 \op \ff^{m_-} \q \T{and} \\
& \wih{d_-}^2 := \ma{cc}{d_-^2 & F_-^2 \\ 0 & 0} : X_-^2 \op \ff^{m_-} \to X_+^2 \op \ff^{m_+}
\end{split}
\]
Apply the notation $\wih D^2 := (\wit X^2, \wih d^2)$. By construction $\wih D^2$ is an exact complex.

Consider also the perturbed differentials on the chains of $\wit D$,
\[
\begin{split}
& \wih{d_+} := \ma{cccc}{d_+^2 & A_- & 0 & F_+^2 \\ 0 & - d_-^1 & 0 & 0 \\ 0 & 0 & 0 & 0 \\ 0 & 0 & 0 & N_+^2} 
: X_+^2 \op X_-^1 \op \ff^{n_-} \op \ff^{m_+} \to X_-^2 \op X_+^1 \op \ff^{n_+} \op \ff^{m_-} \\
& \wih d_- := \ma{cccc}{d_-^2 & A_+ & 0 & F_-^2 \\ 0 & - d_+^1 & 0 & 0 \\ 0 & 0 & 0 & 0 \\ 0 & 0 & 0 & 0} 
: X_-^2 \op X_+^1 \op \ff^{n_+} \op \ff^{m_-} \to X_+^2 \op X_-^1 \op \ff^{n_-} \op \ff^{m_+}
\end{split}
\]
The associated $\zz/2\zz$-graded Fredholm complex $\wih{D} := (\wit D, \wih{d})$ is then chain isomorphic to the mapping cone complex of the chain map $\wih A : \wit D^1 \to \wih D^2$ defined by 
\begin{equation}\label{eq:wihcha}
\begin{split}
& \wih A_+ := (A_+ \op 0) : X^1_+ \op \ff^{n_+}\to X^2_+ \op \ff^{m_+} \q \T{and} \\
& \wih A_- := (A_- \op 0) : X^1_- \op \ff^{n_-} \to X^2_- \op \ff^{m_-}
\end{split}
\end{equation}
In particular we have the following six term exact sequence of homology groups,
\begin{equation}\label{eq:sixyyy}
\begin{CD}
H_+(\wit D^1) @>>> \{0\} @>>> H_+(\wih D) \\
@A{\wih{p_-}}AA & & @VV{\wih{p_+}}V \\
H_-(\wih D) @<<< \{0\} @<<< H_-(\wit D^1)
\end{CD}
\end{equation}
The chain map $\wih p : \wih D \to T \wit D^1$ is given by the projections
\[
\begin{split}
& \wih p_+ : X_+^2 \op X_-^1 \op \ff^{n_-} \op \ff^{m_+} \to X_-^1 \op \ff^{n_-} \q \T{and} \\
& \wih p_- : X_-^2 \op X_+^1 \op \ff^{n_+} \op \ff^{m_-} \to X_+^1 \op \ff^{n_+}
\end{split}
\]
The notation $T(\wih{\G D}) : \ff \to |\wit D^1| \ot |\wih D|$ refers to the torsion isomorphism of the six term exact sequence in \eqref{eq:sixyyy}.

We may apply a similar construction to the complexes $\wit \De^1$, $\wit \De^2$ and $\wit \De$. The $\zz/2\zz$-graded Fredholm complexes $\wih \De^2$ and $\wih \De$ are then finite rank perturbations of $\wih D^2$ and $\wih D$. In particular we have the perturbation isomorphisms
\[
P(\wih \De^2, \wih D^2) : \ff \to \ff \q \T{and} \q
P(\wih \De, \wih D) : |\wih D| \to |\wih \De|
\]
We also have the torsion isomorphism $T(\wih{\G \De}) : \ff \to |\wit \De^1| \ot |\wih \De|$ which is defined in analogy with $T(\wih{\G D})$.

The main result of this subsection can now be stated. It provides a method for reducing the proof of Theorem \ref{t:pertor} to the case where the subcomplexes $D^2$ and $\De^2$ have trivial homology groups.

\begin{prop}\label{p:equdiazer}
The diagram
\begin{equation}\label{eq:diaspeIII}
\begin{CD}
|D^2| @>{T(\G D)}>> |D^1| \ot |D| \\
@V{P(\De^2,D^2)}VV @VV{P(\De^1,D^1) \ot P(\De,D)}V \\
|\De^2| @>>{T(\G \De)}> |\De^1| \ot |\De|
\end{CD}
\end{equation}
commutes if and only if the diagram
\begin{equation}\label{eq:diaspeIV}
\begin{CD}
\ff @>{T(\wih{\G D})}>> |\wit D^1| \ot |\wih D| \\
@V{P(\wih \De^2,\wih D^2)}VV @VV{P(\wit \De^1,\wit D^1) \ot P(\wih \De,\wih D)}V \\
\ff @>>{T(\wih{\G \De})}> |\wit \De^1| \ot |\wih \De|
\end{CD}
\end{equation}
commutes.
\end{prop}

The proof of the above proposition will occupy the rest of this subsection. Let us start by explaining how the proof relies on Proposition \ref{p:equdia}. To this end, remark that $\wih D^2$ and $\wih D$ are finite rank perturbations of $\wit D^2$ and $\wit D$ and similarly that $\wih \De^2$ and $\wih \De$ are finite rank perturbations of $\wit \De^2$ and $\wit \De$. Suppose then that the diagrams
\begin{equation}\label{eq:diaspeV}
\begin{CD}
|\wit D^2| @>{T(\wit{\G D})}>> |\wit D^1| \ot |\wit D| \\ 
@V{P(\wih D^2,\wit D^2)}VV @VV{1 \ot P(\wih D,\wit D)}V \\
\ff @>>{T(\wih{\G D})}> |\wit D^1| \ot |\wih D|
\end{CD} \qq \T{and} \qq
\begin{CD}
|\wit \De^2| @>{T(\wit{\G \De})}>> |\wit \De^1| \ot |\wit \De| \\ 
@V{P(\wih \De^2,\wit \De^2)}VV @VV{1 \ot P(\wih \De,\wit \De)}V \\
\ff @>>{T(\wih{\G \De})}> |\wit \De^1| \ot |\wih \De|
\end{CD}
\end{equation}
commute. An application of Proposition \ref{p:equdia} then yields that the diagram
\[
\begin{CD}
|D^2| @>{T(\G D)}>> |D^1| \ot |D| \\
@V{P(\De^2,D^2)}VV @VV{P(\De^1,D^1) \ot P(\De,D)}V \\
|\De^2| @>>{T(\G \De)}> |\De^1| \ot |\De|
\end{CD}
\]
commutes if and only if the diagram
\[
\begin{CD}
\ff @>{T(\wih{\G D})}>> |\wit D^1| \ot |\wih D| \\
@V{P(\wih \De^2,\wit \De^2)\ci P(\wit \De^2,\wit D^2)\ci P(\wit D^2,\wih D^2)}VV 
@VV{P(\wit \De^1,\wit D^1) \ot \big( P(\wih \De, \wit \De) \ci P(\wit \De,\wit D) \ci P(\wit D,\wih D)\big)}V \\
\ff @>>{T(\wih{\G \De})}> |\wit \De^1| \ot |\wih \De|
\end{CD}
\]
The result of Proposition \ref{p:equdiazer} then follows from the transitivity of the perturbation isomorphisms, see Theorem \ref{t:algprocom}. 

In order to prove Proposition \ref{p:equdiazer} it therefore suffices to show that the diagrams in \eqref{eq:diaspeV} commute. We will focus on the diagram to the left since the commutativity of the diagram to the right follows by the same argumentation.

The first step is to obtain concrete formulas for the perturbation isomorphisms $P(\wih D^2,\wit D^2) : |\wit D^2| \to \ff$ and $P(\wih D,\wit D) : |\wit D| \to |\wih D|$. This step will rely on Proposition \ref{p:pereaz}.

\begin{lemma}\label{l:perv}
The perturbation isomorphism of the finite rank perturbation $\wit D^2 \to \wih D^2$ is given by
\[
\begin{split}
& P(\wih D^2, \wit D^2) : \big|H_+(D^2) \op \ff^{m_+} \big| \ot \big|H_-(D^2) \op \ff^{m_-} \big|^* \to \ff \\
& P(\wih D^2, \wit D^2) : (F^2_- v_- \we v_+ \we z_+ ) \ot (v_- \we F_+^2 v_+ \we N^2_+ z_+)^* \mapsto 1.
\end{split}  
\]
for all non-trivial vectors $v_- \in |V_-^2|$, $v_+ \in |V_+^2|$ and $z_+ \in |Z_+^2|$.
\end{lemma}
\begin{proof}
This follows immediately from Proposition \ref{p:pereaz}.
\end{proof}

The computation of the perturbation isomorphism $P(\wih D,\wit D)$ is more involved. Let us start by choosing a pseudo-inverse
\begin{equation}\label{eq:exapseinv}
\begin{CD}
H_+(D^1) @<<{A_+^\da}< H_+(D^2) @<<{i_+^\da}< H_+(D) \\
@V{p_-^\da}VV & & @AA{p_+^\da}A \\
H_-(D) @>{i_-^\da}>> H_-(D^2) @>{A_-^\da}>> H_-(D^1)
\end{CD}
\end{equation}
of the six term exact sequence of homology groups coming from the mapping cone triangle
\[
\begin{CD}
D^1 @>{A}>> D^2 @>{i}>> D @>{p}>> TD^1
\end{CD}
\]
Choose subspaces $Q_+ \su \T{Ker}(d_+)$ and $Q_- \su \T{Ker}(d_-)$ such that the quotient maps
\[
Q_+ \to H_+(D) \q \T{and} \q
Q_- \to H_-(D)
\]
are injective with images $\T{Im}(p_+^\da)$ and $\T{Im}(p_-^\da)$, respectively.

We may then choose subspaces $V_+ \su V_+^2$ and $V_- \su V_-^2$ such that
\[
V_+^2 = \T{Ker}(i_- F_+^2|_{V_+^2}) \dop V_+ \q \T{and} \q V_-^2 = \T{Ker}(i_+ F_-^2|_{V_-^2}) \dop  V_-
\]
where $i_- F_+^2 : \ff^{m_+} \to H_-(D)$ and $i_+ F_-^2 : \ff^{m_-} \to H_+(D)$. We put $W_+ := \ff^{n_-} \op \T{Ker}(i_- F_+^2|_{V_+^2})$ and $W_- := \ff^{n_+} \op \T{Ker}(i_+ F_-^2|_{V_-^2})$. It is clear that
\[
\ff^{n_-} \op \ff^{m_+} = W_+ \dop V_+ \dop Z_+^2 \q \T{and} \q \ff^{n_+} \op \ff^{m_-} = W_- \dop V_- \dop Z_-^2
\]
Furthermore, we have that
\[
H_+(D) = \T{Im}\big( i_+ F_-^2|_{V_-} \big) \dop \T{Im}(p_+^\da) \q \T{and} \q
H_-(D) = \T{Im}\big( i_- F_+^2|_{V_+} \big) \dop \T{Im}(p_-^\da)
\]
It follows that
\[
0 = \T{dim}\big( H_+(\wit D)\big) - \T{dim}\big( H_-(\wit D) \big) = \T{dim}(Q_+) + \T{dim}(W_+) - \T{dim}(Q_-) - \T{dim}(W_-)
\]
In particular, we may choose an isomorphism $L : \T{Im}(p_+^\da) \op W_+ \to \T{Im}(p_-^\da) \op W_-$ and extend it to and isomorphism $L : H_+(\wit D) \to H_-(\wit D)$ by letting
\[
L : [ i_+ F_-^2 v_-, v_+,z_+] \mapsto [i_- F_+^2 v_+, v_-,N_+ z_+] \q \T{for all } v_- \in V_- \, , \, v_+ \in V_+ \, , \, z_+ \in Z_+^2 
\] 

Let us also choose an isomorphism $M : H_-(\wih D) \to H_+(\wih D)$ such that the diagram
\[
\begin{CD}
H_-(\wih D) @>\phi_->> H_-(\wit D) \\
@V{M}VV @VV{L^{-1}}V \\
H_+(\wih D) @>\phi_+>> H_+(\wit D)
\end{CD}
\]
commutes, where the linear maps $\phi_-$ and $\phi_+$ are defined as in Lemma \ref{l:injhom}. The next lemma is now a consequence of Proposition \ref{p:pereaz}.

\begin{lemma}\label{l:perx}
The perturbation isomorphism associated with the perturbation $\wit D \to \wih D$ is given by
\[
\begin{split}
& P(\wih D, \wit D) : \big| H_+(\wit D) \big| \ot \big| H_-(\wit D) \big|^* \to \big| H_+(\wih D) \big| \ot \big| H_-(\wih D) \big|^* \\
& P(\wih D, \wit D) : s_+ \ot (L s_+)^* \mapsto (M  s_-) \ot s_-^*
\end{split}
\]
for all non-trivial vectors $s_+ \in \big| H_+(\wit D) \big|$ and $s_- \in \big| H_-(\wih D)\big|$.
\end{lemma}

We are now ready to prove that the diagrams in \eqref{eq:diaspeV} commute. As noted above, this implies the result of Proposition \ref{p:equdiazer}.

\begin{lemma}\label{l:diaindzer}
The diagram
\[
\begin{CD}
|\wit D^2| @>{T(\wit{\G D})}>> |\wit D^1| \ot |\wit D| \\
@V{P(\wih D^2,\wit D^2)}VV @VV{1 \ot P(\wih D,\wit D)}V \\
\ff @>{T(\wih{\G D})}>> |\wit D^1| \ot |\wih D|
\end{CD}
\]
is commutative.
\end{lemma}
\begin{proof}
Choose non-trivial vectors
\[
\begin{split}
& t^1_+ \in \big| \T{Im}(A_+^\da) \big| \q t^2_+ \in \big| \T{Im}(i_+^\da) \big| \q t_+ \in \big| \T{Im}(p_+^\da) \big| \q  \T{and} \\
& t^1_- \in \big| \T{Im}(A_-^\da)\big| \q t^2_- \in \big| \T{Im}(i_-^\da) \big| \q t_- \in \big| \T{Im}(p_-^\da) \big| 
\end{split}
\]
Let us also choose non-trivial vectors $v_+ \in |V_+|$, $v_- \in |V_-|$ and $w_+ \in \big|\T{Ker}(i_- F_+^2|_{V_+^2})\big|$, $w_- \in \big|\T{Ker}(i_+ F_-^2|_{V_-^2}) \big|$ such that
\[
F_+^2 v_+ = t_-^2 \, \, , \, \, F_-^2 v_- = t_+^2 \, \, \, \T{ and } \, \, \, F_+^2 w_+ = A_-(t_-^1) \, \, , \, \, F_-^2 w_- = A_+(t_+^1)
\]
Notice here that we may assume, without loss of generality, that $\T{Im}(F_+^2|_{V_+}) = \T{Im}(i_-^\da)$ and $\T{Im}(F_-^2|_{V_-}) = \T{Im}(i_+^\da)$. Finally, choose non-trivial vectors
\[
z_+ \in |Z_+^2| \, \, \, \T{ and } \, \, \, \eta_+ \in |\ff^{n_+}| \, \, , \, \, \eta_- \in |\ff^{n_-}|
\]
To ease the notation, let $\xi_+ := v_+ \we z_+ \we w_+ \in |\ff^{m_+}|$, $\xi_- := v_- \we N_+^2 z_+ \we w_- \in |\ff^{m_-}|$. 

The torsion isomorphism $T(\wit{\G D})$ is then given by
\begin{equation}\label{eq:wittor}
\begin{split}
& T(\wit{\G D}) : (F_-^2 w_- \we F_-^2 v_- \we \xi_+ ) \ot ( F_+^2 w_+ \we F_+^2 v_+ \we \xi_- )^* \\
& \qq \mapsto (-1)^{\mu(\wit{\G D})} \cd (p_- t_- \we \eta_+ \we t^1_+) \ot ( p_+ t_+ \we \eta_- \we t^1_-)^* \\
& \qqqq \ot ( i_+(t^2_+) \we \xi_+ \we t_+ \we \eta_- ) \ot ( i_-(t^2_-) \we \xi_- \we t_- \we \eta_+ )^*
\end{split}
\end{equation}
where the sign exponent $\mu(\wit{\G D}) \in \nn_0$ has the expression
\[
\begin{split}
\mu(\wit{\G D}) & = \big(\ep(t^2_+) + m_+ + 1\big)\cd \big( \ep(t^1_-) + \ep(t^1_+) \big) + \ep(t^1_-) \cd \big( \ep(t_+) + \ep(t_-) + n_+ + n_-\big) \\ 
& \q + \big(\ep(t_-) + n_+ \big)\cd \big(\ep(t^2_+) + \ep(t^2_-) + m_+ + m_-\big) + \ep(t_+) + n_-
\end{split}
\]
See Definition \ref{d:toriso}.

For the sake of simplicity, let 
\[
\begin{split}
& \ka_+ := w_+ \we t_+ \we \eta_- \in \big| \T{Im}(p_+^\da) \op \ff^{n_-} \op \T{Ker}(i_- F_+^2|_{V_+^2}) \big|
\q \T{and} \\
& \ka_- := w_- \we t_- \we \eta_+ \in \big| \T{Im}(p_-^\da) \op \ff^{n_+} \op \T{Ker}(i_+ F_-^2|_{V_-^2}) \big|
\end{split}
\]
There is then a unique constant $\la \in \ff^*$ such that
\[
\la \cd \ka_- = L(\ka_+)
\]
It follows that
\[
\begin{split}
& L\big( i_+(t^2_+) \we \xi_+ \we t_+ \we \eta_- \big) = L\big( i_+(t^2_+) \we v_+ \we z_+ \we \ka_+ \big) \\
& \q = \la \cd \big( v_- \we (i_- F_+^2 v_+) \we N_+^2 z_+ \we \ka_- \big)
= \la \cd (-1)^{\ep(t^2_-) \cd \ep(t^2_+)} \big( i_- t^2_- \we \xi_- \we t_- \we \eta_+ \big)
\end{split}
\]
An application of Lemma \ref{l:perx} then yields that
\[
\begin{split}
& \big( ( 1 \ot P(\wih D,\wit D) ) \ci T(\wit{\G D}) \big)
\big( (F_-^2 w_- \we F_-^2 v_- \we \xi_+ ) \ot ( F_+^2 w_+ \we F_+^2 v_+ \we \xi_- )^* \big) \\
& \q = \la \cd (-1)^{\mu(\wit{\G D}) + \ep(t_+^2) \cd \ep(t_-^2)}\cd (p_- t_- \we \eta_+ \we t^1_+) \ot ( p_+ t_+ \we \eta_- \we t^1_-)^* \ot
M r_- \ot r_-^*
\end{split}
\]
for any non-zero vector $r_- \in \big| H_-(\wih D) \big|$.

Without loss of generality, we may suppose that $\wih{p_-}(r_-) = p_- t_- \we \eta_+ \we t_+^1 \in \big| H_+(\wit D^1)\big|$. Let us also choose $r_+ \in \big| H_+(\wih D) \big|$ such that $\wih{p_+}(r_+) = p_+ t_+ \we \eta_- \we t_-^1 \in \big| H_-(\wit D^1)\big|$.

The torsion isomorphism $T(\G{\wih D}) : \ff \to |\wit D^1| \ot |\wih D|$ is then given by
\begin{equation}\label{eq:wihtor}
T(\G{\wih D}) : 1 \mapsto (-1)^{\T{dim}(H_+(\wih D))} \cd (\wih{ p_-} r_-)  \ot (\wih{ p_+} r_+)^* \ot  r_+ \ot (r_-)^*
\end{equation}
It then follows from Lemma \ref{l:perv} that
\[
\begin{split}
& \big( T(\G{\wih D}) \ci P(\wih D^2,\wit D^2) \big)\big( 
(F_-^2 w_- \we F_-^2 v_- \we \xi_+ ) \ot ( F_+^2 w_+ \we F_+^2 v_+ \we \xi_- )^* \big) \\
& = (-1)^{\mu(L^2) + \T{dim}(H_+(\wih D))}
\cd (\wih{ p_-} r_-)  \ot (\wih{ p_+} r_+)^* \ot  r_+ \ot (r_-)^*
\end{split}
\]
where the sign exponent is given by
\[
\mu(L^2) := \big(\ep(t_+^1) + \ep(t_-^1) \big) \cd \big( m_+ + \ep(t_-^1) + \ep(t_+^2) \big) + \ep(t_-^1) \cd \ep(t_+^1) + \ep(t_-^2) \cd \ep(t_+^2)
\]
Indeed, a straightforward computation shows that
\[
(-1)^{\mu(L^2)} \cd F_+^2 w_+ \we F_+^2 v_+ \we \xi_- = L^2(F_-^2 w_- \we F_-^2 v_- \we \xi_+ )
\]

It is therefore enough to show that
\[
M r_- = \la^{-1} \cd (-1)^{\mu(\wit{\G D}) + \ep(t_+^2) \cd \ep(t_-^2) + \mu(L^2) + \T{dim}(H_+(\wih D))} \cd r_+
\]
This is the content of the next lemma.
\end{proof}

\begin{lemma}
\[
M r_- = \la^{-1} \cd (-1)^{\mu(\wit{\G D}) + \ep(t_+^2) \cd \ep(t_-^2) + \mu(L^2) + \T{dim}(H_+(\wih D))} \cd r_+
\]
\end{lemma}
\begin{proof}
It suffices to prove that
\[
(L^{-1}\phi_-) r_- = (\phi_+ M) r_- = \la^{-1} \cd (-1)^{\mu(\wit{\G D}) + \ep(t_+^2) \cd \ep(t_-^2) + \mu(L^2) + \T{dim}(H_+(\wih D))} \cd \phi_+(r_+)
\]

However, using the identities $\wih{p_+}(r_+) = p_+ t_+ \we \eta_- \we t_-^1$ and $F_+^2 w_+ = A_-(t_-^1)$ we obtain that
\[
\begin{split}
\phi_+(r_+) & = t_+ \we \eta_- \we (-w_+) = (-1)^{\ep(t_-^1) \cd (\ep(t_+) + n_- + 1)} \cd w_+ \we t_+ \we \eta_- \\
& = (-1)^{\ep(t_-^1) \cd (\ep(t_+) + n_- + 1)} \cd \ka_+
\end{split}
\]
A similar computation implies that
\[
\phi_-(r_-) = (-1)^{\ep(t_+^1) \cd (\ep(t_-) + n_+ + 1)} \cd \ka_-
\]

We thus have that
\[
\begin{split}
(L^{-1} \phi_-) r_- 
& = (-1)^{\ep(t_+^1) \cd (\ep(t_-) + n_+ + 1)} \cd \la^{-1} \cd \ka_+ \\ 
& = (-1)^{\ep(t_+^1) \cd (\ep(t_-) + n_+ + 1) + \ep(t_-^1) \cd (\ep(t_+) + n_- + 1)} \cd \la^{-1} \cd \phi_+(r_+)
\end{split}
\]

It is therefore sufficient to show that
\[
(-1)^{\mu(\wit{\G D}) + \ep(t_+^2) \cd \ep(t_-^2) + \mu(L^2) + \T{dim}(H_+(\wih D)) + 
\ep(t_+^1) \cd (\ep(t_-) + n_+ + 1) + \ep(t_-^1) \cd (\ep(t_+) + n_- + 1) } = 1
\]

Now, using that $\T{Ind}(\wit D^1) = \T{Ind}(\wit D^2) = \T{Ind}(\wit D) = 0$ we obtain that
\[
(-1)^{\mu(\wit{\G D})} = (-1)^{\big( \ep(t_+^2) + m_+ + \ep(t_-^1) + n_+ + \ep(t_-) + 1\big) \cd \big( \ep(t_-^1) + \ep(t_+^1) \big) + \ep(t_+) + n_-}
\]
Similarly, we get that
\[
(-1)^{\ep(t_+^1) \cd (\ep(t_-) + n_+ + 1) + \ep(t_-^1) \cd (\ep(t_+) + n_- + 1)}
= (-1)^{\big( \ep(t_+^1) + \ep(t_-^1) \big) \cd \big( \ep(t_-) + n_+ + \ep(t_-^1) + 1\big) }
\]
Combining these observations with the expression for $\mu(L^2)$ we deduce that
\[
\begin{split}
& (-1)^{\mu(\wit{\G D}) + \ep(t_+^2) \cd \ep(t_-^2) + \mu(L^2) + \T{dim}(H_+(\wih D)) + 
\ep(t_+^1) \cd (\ep(t_-) + n_+) + \ep(t_-^1) \cd (\ep(t_+) + n_-) } \\
& \q = (-1)^{\ep(t_+) + n_- + \T{dim}(H_+(\wih D)) + \ep(t_-^1)}
= 1
\end{split}
\]
This proves the lemma.
\end{proof}

\subsection{Index zero and exact subcomplex}
In this subsection we will prove a simplified version of Theorem \ref{t:pertor}. \emph{It will thus be a standing assumption that the $\zz/2\zz$-graded Fredholm complex $D^1$ has index zero and that $D^2$ and $\De^2$ have trivial homology groups.}

We start with a preliminary lemma on the structure of pseudo-inverses. 

\begin{lemma}\label{l:psetri}
Let $V_+$, $W_+$, $V_-$ and $W_-$ be vector spaces over $\ff$ and let $T = \ma{cc}{A & B \\ 0 & D} : V_+ \op W_+ \to V_- \op W_-$ be a linear map. Suppose that $B(\T{Ker}(D)) \su \T{Im}(A)$ and let 
\[
D^\da : W_- \to W_+ \q \T{and} \q A^\da : V_- \to V_+
\]
be pseudo-inverses of $D : W_+ \to W_-$ and $A : V_+ \to V_-$. Then the upper triangular matrix
\[
T^\da = \ma{cc}{A^\da & -A^\da B D^\da \\ 0 & D^\da} : V_- \op W_- \to V_+ \op W_+
\]
provides a pseudo-inverse of $T : V_+ \op W_+ \to V_- \op W_-$.
\end{lemma}
\begin{proof}
We need to verify the two identities
\begin{equation}\label{eq:pseide}
T T^\da T = T \q \T{and} \q T^\da T T^\da = T^\da
\end{equation}

A straightforward computation shows that
\[
T T^\da T = \ma{cc}{A & A A^\da B (1 - D^\da D )  + B D^\da D \\ 
0 & D}
\]
Now, notice that the assumption $B(\T{Ker}(D)) \su \T{Im}(A)$ implies that
\[
A A^\da B (1 - D^\da D ) = B (1 - D^\da D)
\]
Combining this identity with the above computation we conclude that $T T^\da T = T$.

The second identity in \eqref{eq:pseide} follows by a direct computation. It does not depend on the assumption $B(\T{Ker}(D)) \su \T{Im}(A)$.
\end{proof}

Let us now choose pseudo-inverses
\[
\begin{CD}
(D^1)^\da : X_-^1 @>{(d^1_+)^\da}>> X_+^1 @>{(d^1_-)^\da}>> X_-^1 \q , \q
(\De^1)^\da : X_-^1 @>{(\de^1_+)^\da}>> X_+^1 @>{(\de^1_-)^\da}>> X_-^1 \\
(D^2)^\da : X_-^2 @>{(d^2_+)^\da}>> X_+^2 @>{(d^2_-)^\da}>> X_-^2 \q , \q
(\De^2)^\da : X_-^2 @>{(\de^2_+)^\da}>> X_+^2 @>{(\de^2_-)^\da}>> X_-^2
\end{CD}
\]
of the $\zz/2\zz$-graded Fredholm complexes $D^1$, $\De^1$ and $D^2$, $\De^2$ such that $(D^1)^\da \to (\De^1)^\da$ and $(D^2)^\da \to (\De^2)^\da$ are finite rank perturbations.

Since the homology groups of $D^2$ and $\De^2$ are trivial it follows from Lemma \ref{l:psetri} that
\[
\begin{CD}
D^\da : X_- @>{d_+^\da}>> X_+ @>{d_-^\da}>> X_- \q \T{and} \q
\De^\da : X_- @>{\de_+^\da}>> X_+ @>{\de_-^\da}>> X_-
\end{CD}
\]
are pseudo-inverses of $D$ and $\De$ where
\[
\begin{split}
& d_+^\da := \ma{cc}{(d_+^2)^\da & (d_+^2)^\da A_- (d_-^1)^\da \\ 0 & -(d_-^1)^\da} \q , \q
d_-^\da := \ma{cc}{(d_-^2)^\da & (d_-^2)^\da A_+ (d_+^1)^\da \\ 0 & -(d_+^1)^\da} \\
& \de_+^\da := \ma{cc}{(\de_+^2)^\da & (\de_+^2)^\da B_- (\de_-^1)^\da \\ 0 & -(\de_-^1)^\da} \q , \q
\de_-^\da := \ma{cc}{(\de_-^2)^\da & (\de_-^2)^\da B_+ (\de_+^1)^\da \\ 0 & -(\de_+^1)^\da}
\end{split}
\]
Remark also that $D^\da \to \De^\da$ is a finite rank perturbation.

We will denote the Fredholm operators associated with the $\zz/2\zz$-graded Fredholm complexes containing the letter ``$D$'' by $\si_+^1$, $\si_-^1$, etc. The idempotents will be denoted by $Q_+^1$, $Q_-^1$, etc. The Fredholm operators and idempotents coming from the complexes with ``$\De$'' will be denoted by $\tau_+^1$, $\tau_-^1$, etc. and $\Pi_+^1$, $\Pi_-^1$, etc.

Since the $\zz/2\zz$-graded Fredholm complexes $D^1$ and $\De^1$ have index zero, we may choose isomorphisms
\begin{equation}\label{eq:perquocom}
L^1 : \T{Ker}(\si_+^1) \to \T{Ker}(\si_-^1) \q \T{and} \q
M^1 : \T{Ker}(\tau_-^1) \to \T{Ker}(\tau_+^1)
\end{equation}
We then have the isomorphism of determinant class $\Si^1 := (\tau_-^1 + M^1 \Pi_-^1)(\si_+^1 + L^1 Q_+^1) : X_+^1 \to X_+^1$.

Since the $\zz/2\zz$-graded Fredholm complexes $D^2$ and $\De^2$ are exact, we have the isomorphism of determinant class $\Si^2 := \tau_-^2 \si_+^2 : X_+^2 \to X_+^2$.

In order to obtain a good expression for the isomorphism of determinant class $\Si : X_+ \to X_+$, we remark that the projections $p_+ : X_+^2 \op X_-^1 \to X_-^1$ and $p_- : X_-^2 \op X_+^1 \to X_+^1$ induce isomorphisms
\[
\begin{split}
& p_+ : \T{Ker}(\si_+) \to \T{Ker}(\si_-^1) \, \, \, , \, \, \, p_- : \T{Ker}(\si_-) \to \T{Ker}(\si_+^1) \q \T{and} \\
& p_+ : \T{Ker}(\tau_+) \to \T{Ker}(\tau_-^1) \, \, \, , \, \, \, p_- : \T{Ker}(\tau_-) \to \T{Ker}(\tau_+^1)
\end{split}
\]
This is again a consequence of the assumption that $D^2$ and $\De^2$ have trivial homology groups. We may thus define the isomorphisms
\begin{equation}\label{eq:permapcom}
\begin{split}
& L := - (p_-)^{-1} (L^1)^{-1} p_+ : \T{Ker}(\si_+) \to \T{Ker}(\si_-) \q \T{and} \\
& M := - (p_+)^{-1} (M^1)^{-1} p_- : \T{Ker}(\tau_-) \to \T{Ker}(\tau_+)
\end{split}
\end{equation}
The isomorphism of determinant class is then given by $\Si := (\tau_- + M \Pi_-)(\si_+ + L Q_+)$.

\begin{lemma}\label{l:thrdetpro}
\[
\T{det}(\Si^2) = \T{det}(\Si) \cd \T{det}(\Si^1)
\]
\end{lemma}
\begin{proof}
The invertible maps $\tau_- + M \Pi_- : X_-^2 \op X_+^1 \to X_+^2 \op X_-^1$ and $\si_+ + L Q_+ : X_+^2 \op X_-^1 \to X_-^2 \op X_+^1$ are both upper triangular. The diagonals are given by
\[
\ma{cc}{\tau_-^2 & \ast \\ 0 & - \tau_+^1 - (M^1)^{-1} \Pi_+^1} \q \T{and} \q
\ma{cc}{\si_+^2 & \ast \\ 0 & - \si_-^1 - (L^1)^{-1} Q_-^1}
\]
It follows that $\Si$ is upper triangular with diagonal
\[
\ma{cc}{\Si^2 & \ast \\ 0 & (\tau_+^1 + (M^1)^{-1} \Pi_+^1)(\si_-^1 + (L^1)^{-1} Q_-^1)}
\]
It is therefore enough to show that
\begin{equation}\label{eq:idedetexa}
\begin{split}
& \T{det}\big( (\tau_+^1 + (M^1)^{-1} \Pi_+^1)(\si_-^1 + (L^1)^{-1} Q_-^1)\big) \cd \T{det}(\Si^1) \\
& \q = \T{det}\big( (\tau_+^1 + (M^1)^{-1} \Pi_+^1)(\si_-^1 + (L^1)^{-1} Q_-^1)\big) \cd \T{det}\big( (\tau_-^1 + M^1 \Pi_-^1)(\si_+^1 + L^1 Q_+^1 ) \big) \\
& \q = 1
\end{split}
\end{equation}
However, we have that
\[
\begin{split}
& (\si_-^1 + (L^1)^{-1} Q_-^1)(\si_+^1 + L^1 Q_+^1) = \si_-^1 \si_+^1 + Q_+^1 = 1 \q \T{and} \\
& (\tau_+^1 + (M^1)^{-1} \Pi_+^1)(\tau_-^1 + M^1 \Pi_-^1) = \tau_+^1 \tau_-^1 + \Pi_-^1 = 1
\end{split}
\]
This clearly implies the identity in \eqref{eq:idedetexa}.
\end{proof}

We are now ready to prove the main result of this subsection. It shows that the statement of Theorem \ref{t:pertor} is valid in the more restrictive setting of this subsection.

\begin{prop}\label{p:pertorexa}
Suppose that the index of $D^1$ is trivial and that $D^2$ and $\De^2$ are exact. Then the diagram
\[
\begin{CD}
\ff @>{T(\G D)}>> |D^1| \ot |D| \\
@V{P(\De^2,D^2)}VV @VV{P(\De^1,D^1) \ot P(\De,D)}V \\
\ff @>>{T(\G \De)}> |\De^1| \ot |\De|
\end{CD}
\]
commutes.
\end{prop}
\begin{proof}
For any non-trivial vectors $t_+ \in |H_+(D)|$, $t_- \in |H_-(D)|$ and $s_+ \in |H_+(\De)|$, $s_- \in |H_-(\De)|$ we have that
\[
\begin{split}
T(\G D)(1) & = p_- t_- \ot (p_+ t_+)^* \ot t_+ \ot t_-^* \cd (-1)^{\T{dim}(H_+(D))} \q \T{and} \\
T(\G \De)(1) & = p_- s_- \ot (p_+ s_+)^* \ot s_+ \ot s_-^* \cd (-1)^{\T{dim}(H_+(\De))}
\end{split}
\]
Let now $L : H_+(D) \to H_-(D)$, $M : H_-(\De) \to H_+(\De)$ and $L^1 : H_+(D^1) \to H_-(D^1)$, $M^1 : H_-(\De^1) \to H_+(\De^1)$ be the isomorphisms of homology groups induced by the maps in \eqref{eq:permapcom} and \eqref{eq:perquocom}, respectively.

Letting $t_- := L t_+$ and noting that $p_- t_- = (-1)^{\ep(t_+)} \cd (L^1)^{-1} p_+ t_+$ it follows that
\[
\begin{split}
& \big( ( P(\De^1,D^1) \ot P(\De,D) ) \ci T(\G D) \big)(1) \\
& \q = (-1)^{\ep(t_+) + \T{dim}(H_+(D))} \cd \T{det}(\Si^1)^{-1} \cd \T{det}(\Si)^{-1} \\ 
& \qqq \cd M^1 p_+ Ms_- \ot (p_+ Ms_-)^* \ot Ms_- \ot s_-^*
\end{split}
\]
Now, by Lemma \ref{l:thrdetpro} we have that $\T{det}(\Si^1)^{-1} \cd \T{det}(\Si)^{-1} = \T{det}(\Si^2)^{-1}$. Furthermore, we clearly have that $p_+ M s_- = (-1)^{\ep(s_-)} \cd (M^1)^{-1} p_- s_-$.

This shows that
\[
\begin{split}
& \big( ( P(\De^1,D^1) \ot P(\De,D) ) \ci T(\G D) \big)(1) \\
& \q = (-1)^{\ep(t_+) + \T{dim}(H_+(D)) + \ep(s_-)} \cd \T{det}(\Si^2)^{-1} \cd p_- s_- \ot (p_+ Ms_-)^* \ot Ms_- \ot s_-^* \\
& \q = (-1)^{\ep(t_+) + \T{dim}(H_+(D)) + \ep(s_-) + \T{dim}(H_+(\De))} \cd \T{det}(\Si^2)^{-1} \cd T(\De)(1)
\end{split}
\]
The result of the lemma thus follows by noting that 
\[
(-1)^{\ep(t_+) + \T{dim}(H_+(D)) + \ep(s_-) + \T{dim}(H_+(\De))} = 1
\]
\end{proof}

\subsection{Proof of Theorem \ref{t:pertor}}
In this subsection we will recollect the results of the previous subsections and thereby provide a proof of Theorem \ref{t:pertor}. 

First of all, we notice that the result of Proposition \ref{p:equdiazer} implies that it suffices to prove the commutativity of the diagram
\begin{equation}\label{eq:diaspeVI}
\begin{CD}
\ff @>{T(\wih{\G D})}>> |\wit D^1| \ot |\wih D| \\
@V{P(\wih \De^2,\wih D^2)}VV @VV{P(\wit \De^1,\wit D^1) \ot P(\wih \De,\wih D)}V \\
\ff @>>{T(\wih{\G \De})}> |\wit \De^1| \ot |\wih \De|
\end{CD}
\end{equation}

Consider now the chain maps
\[
\wih A : \wit D^1 \to \wih D^2 \q \T{and} \q \wih B : \wit \De^1 \to \wih \De^2
\]
as defined in \eqref{eq:wihcha}. It then follows by Proposition \ref{p:pertorexa} that the diagram
\begin{equation}\label{eq:diaspeVII}
\begin{CD}
\ff @>{T(\wih A)}>> |\wit D^1| \ot |C^{\wih A}| \\
@V{P(\wih \De^2,\wih D^2)}VV @VV{P(\wit \De^1,\wit D^1) \ot P(C^{\wih B},C^{\wih A})}V \\
\ff @>>{T(\wih B)}> |\wit \De^1| \ot |C^{\wih B}|
\end{CD}
\end{equation}
commutes, where $T(\wih A) : \ff \to |\wit D^1| \ot |C^{\wih A}|$ and $T(\wih B) : \ff \to |\wit \De^1| \ot |C^{\wih B}|$ are the torsion isomorphisms coming from the mapping cone triangles
\[
\begin{CD}
\wit D^1 @>{\wih A}>> \wih D^2 @>{i}>> C^{\wih A} @>{p}>> T \wit D^1 & & \q \T{and} \\
\wit \De^1 @>{\wih B}>> \wih \De^2 @>{i}>> C^{\wih B} @>{p}>> T \wit \De^1 & & 
\end{CD}
\]
It therefore suffices to show that the diagram in \eqref{eq:diaspeVI} commutes if and only if the diagram in \eqref{eq:diaspeVII} commutes.

To this end, we define the isomorphisms of vector spaces
\[
\begin{split}
& \Phi_+ : X_+^2 \op \ff^{m_+} \op X_-^1 \op \ff^{n_-} \to X_+^2 \op X_-^1 \op \ff^{n_-} \op \ff^{m_+} \q \T{and} \\
& \Phi_- : X_-^2 \op \ff^{m_-} \op X_+^1 \op \ff^{n_+} \to X_-^2 \op X_+^1 \op \ff^{n_+} \op \ff^{m_-}
\end{split}
\]
which interchanges the factors in the direct sums. These two isomorphisms then provide us with chain isomorphisms
\[
\Phi : C^{\wih A} \to \wih D \q \T{and} \q \Phi : C^{\wih B} \to \wih \De
\]
It is then not hard to see that the diagrams
\[
\begin{CD}
\ff @>{T(\wih A)}>> |\wit D^1| \ot |C^{\wih A}| \\ 
@V{1}VV @VV{1 \ot |\Phi|}V \\
\ff @>>{T(\wih{\G D})}> |\wit D^1| \ot |\wih D|
\end{CD} \qq \T{and} \qq
\begin{CD}
\ff @>{T(\wih B)}>> |\wit \De^1| \ot |C^{\wih B}| \\
@V{1}VV @VV{1 \ot |\Phi|}V \\
\ff @>>{T(\wih{\G \De})}> |\wit \De^1| \ot |\wih \De|
\end{CD}
\]
commute. But this implies that the diagram in \eqref{eq:diaspeVI} commutes if and only if the diagram
\[
\begin{CD}
\ff @>{T(\wih A)}>> |\wit D^1| \ot |C^{\wih A}| \\
@V{P(\wih \De^2,\wih D^2)}VV @VV{P(\wit \De^1,\wit D^1) \ot \big(|\Phi|^{-1} \ci P(\wih \De ,\wih D) \ci |\Phi|\big)}V \\
\ff @>>{T(\wih B)}> |\wit \De^1| \ot |C^{\wih B}|
\end{CD}
\]
commutes. The result of Theorem \ref{t:pertor} is thus proved since it follows by a straightforward verification that
\[
|\Phi|^{-1} \ci P(\wih \De ,\wih D) \ci |\Phi| = P(C^{\wih B}, C^{\wih A}) : |C^{\wih A}| \to |C^{\wih B}|
\]

\section{Parametrized perturbation isomorphisms}\label{s:parperiso}
\emph{Throughout this section $k \in \nn$ will be fixed and $U \su \cc^k$ will be a connected open set.}

\begin{dfn}
A chain complex of Hilbert spaces is a chain complex $D := (X,d)$ such that each $X_j$ is a Hilbert space and each differential $d_j : X_j \to X_{j-1}$ is a bounded operator. 

A chain map $A : D \to E$ between to chain complexes of Hilbert spaces $D := (X,d)$ and $E := (Y,\de)$ is a chain map $A : D \to E$ such that $A_j : X_j \to Y_j$ is a bounded operator for each $j \in \zz$.

A chain complex of Hilbert spaces $D$ is \emph{Fredholm} when the homology group $H_j(D)$ has finite dimension for all $j \in \zz$.
\end{dfn}

Remark that our chain complexes are always assumed to be \emph{bounded} in the sense that there exists a $J \in \nn$ such that $X_j = \{0\}$ whenever $|j| \geq J$.

For each element $z \in U$, let $D^z := (X,d^z)$ be a chain complex of Hilbert spaces. Notice that the Hilbert spaces $X_j$ do not depend on the parameter $z \in U$. 

\begin{dfn}
The family $D := \{ D^z\}_{z \in U}$ of chain complexes of Hilbert spaces is a \emph{holomorphic family} when the associated maps $U \to \sL(X_j,X_{j-1})$, $z \mapsto d_j^z$ are holomorphic with respect to the operator norm for all $j \in \zz$.

A chain map $A : D \to \De$ between two holomorphic families $D := \{ (X,d^z)\}_{z \in U}$ and $\De := \{ (Y,\de^z)\}_{z \in U}$ is given by a chain map $A^z : D^z \to \De^z$ for each $z \in U$ such that the maps $A_j : U \to \sL(X_j,Y_j)$, $z \mapsto A^z_j$ are holomorphic in operator norm for all $j \in \zz$.
\end{dfn}

\emph{From now on, $D = \{ D^z\}_{z \in U}$ and $\De = \{ \De^z \}_{z \in U}$ will be holomorphic families of $\zz$-graded chain complexes of Hilbert spaces. It will be a standing assumption that $D^z := (X,d^z)$ and $\De^z := (Y,\de^z)$ are \emph{exact} for all $z \in U$. Furthermore, we will assume the existence of trivial Fredholm complexes $C$ and $\Ga$ of \emph{index zero} such that $D^z \op C$ and $\De^z \op \Ga$ are \emph{finite rank perturbations} of each other for all $z \in U$.}

For each $z \in U$, let $i_C : \cc \to |C| = |D^z \op C|$ denote the isomorphism defined by $i_C : 1 \mapsto \om \ot \om^*$, where $\om \in \T{det}(C_+) = \T{det}(C_-)$ is a non-trivial vector. The isomorphisms $i_\Ga : \cc \to |\Ga| = |\De^z \op \Ga|$ are defined in a similar way.
 
\emph{The main aim of this section is to study the analyticity of the associated perturbation isomorphisms:}
\[
i_\Ga^{-1} P(\De^z \op \Ga, D^z \op C) i_C : U \to \cc^*
\]

The next lemma will play an important role and we therefore present a detailed proof of it. The notation $\sL^1(H,G)$ will refer to the Banach space of bounded operators $T : H \to G$ with $\T{Tr}(|T|) < \infty$, where $\T{Tr} : \sL(H)_+ \to [0,\infty]$ is the operator trace. The norm $\|\cd \|_1$ on $\sL^1(H,G)$ is defined by $\|T\|_1 := \T{Tr}(|T|)$ for all $T \in \sL^1(H,G)$. Remark that each bounded operator of finite rank $T : H \to G$ defines an element in $\sL^1(H,G)$.

\begin{lemma}\label{l:holidetra}
Let $H$ and $G$ be Hilbert spaces and let $A,B : U \to \sL(H,G)$ be holomorphic maps such that $A^z$ and $B^z$ have closed images for all $z \in U$. Suppose that there exist holomorphic maps $E,\Phi : U \to \sL(G)$ such that $E^z$ and $\Phi^z$ are idempotents with $\T{Im}(E^z) = \T{Im}(A^z)$ and $\T{Im}(\Phi^z) = \T{Im}(B^z)$ for all $z \in U$. Suppose also that $A^z - B^z$ has finite rank for all $z \in U$ and that the associated map $A - B : U \to \sL^1(H,G)$ is holomorphic.

Let $z_0 \in U$. Then there exist an open neighborhood $V \su U$ of $z_0$ and holomorphic maps $K : V \to \sL(H)$ and $\Om : V \to \sL(H)$ such that $K^z$ and $\Om^z$ are idempotents with $\T{Im}(K^z) = \T{Ker}(A^z)$ and $\T{Im}(\Om^z) = \T{Ker}(B^z)$ for all $z \in V$. Furthermore, we may arrange that the difference $K^z - \Om^z$ has finite rank for all $z \in V$ and that the map
$K - \Om : V \to \sL^1(H)$ is holomorphic.
\end{lemma}
\begin{proof}
There exists an open neighborhood $V \su U$ of $z_0 \in U$ such that
\[
E^{z_0} : \T{Im}(A^z) \to \T{Im}(A^{z_0}) \q \T{and} \q \Phi^{z_0} : \T{Im}(B^z) \to \T{Im}(B^{z_0})
\]
are isomorphisms of Hilbert spaces for all $z \in V$. This implies that $\T{Ker}(A^z) = \T{Ker}(E^{z_0} A^z)$ and $\T{Ker}(B^z) = \T{Ker}(\Phi^{z_0} B^z)$ for all $z \in V$. Without loss of generality we may thus suppose that $\T{Im}(A^z) = \T{Im}(A^{z_0})$ and $\T{Im}(B^z) = \T{Im}(B^{z_0})$ for all $z \in U$.

Since $A^{z_0} : \T{Ker}(A^{z_0})^\perp \to \T{Im}(A^{z_0})$ and $B^{z_0} : \T{Ker}(B^{z_0})^\perp \to \T{Im}(B^{z_0})$ are isomorphisms of Hilbert spaces, there exists an open neighborhood $V \su U$ of $z_0 \in U$ such that
\[
A^z : \T{Ker}(A^{z_0})^\perp \to \T{Im}(A^{z_0}) \q \T{and} \q
B^z : \T{Ker}(B^{z_0})^\perp \to \T{Im}(B^{z_0})
\]
are isomorphisms for all $z \in V$. Let $\al^z : \T{Im}(A^{z_0}) \to \T{Ker}(A^{z_0})^\perp$ and $\be^z : \T{Im}(B^{z_0}) \to \T{Ker}(B^{z_0})^\perp$ denote the inverses. 

Define the bounded idempotents $K^z := 1 - \al^z A^z : H \to H$ and $\Om^z := 1 - \be^z B^z : H \to H$ for all $z \in V$. It is then clear that $K,\Om : V \to \sL(H)$ are holomorphic and that $\T{Im}(K^z) = \T{Ker}(A^z)$ and $\T{Im}(\Om^z) = \T{Ker}(B^z)$ for all $z \in V$.

Compute now as follows,
\begin{equation}\label{eq:holidedif}
\begin{split}
K^z - \Om^z & = \be^z B^z - \al^z A^z = \be^z B^z - \be^z B^z \al^z A^z - \Om^z \al^z A^z \\
& = \be^z (B^z - A^z) K^z - \Om^z \al^z A^z
\end{split}
\end{equation}
In order to prove that $K^z - \Om^z$ has finite rank for all $z \in V$ and that $K - \Om : V \to \sL^1(H)$ is holomorphic, it is therefore enough to show that each of the maps
\[
\be (B - A) K  \, \, \T{ and }  \, \, \Om \al A : V \to \sL(H)
\]
has this property. This can be verified immediately for the map $\be (B - A) K : V \to \sL(H)$. 

To prove the claim for the map $\Om \al A : V \to \sL(H)$, it suffices to show that there exists an $m \in \nn$ such that $\T{dim}\big( \T{Im}(\Om^z \al^z A^z) \big) \leq m$ for all $z \in V$. To this end, we remark that the image of the adjoint $(A^{z_0})^* : G \to H$ is closed since the image of $A^{z_0} : H \to G$ is closed by assumption. We thus have that $\T{Im}(\Om^z \al^z A^z) = \T{Im}\big(\Om^z (A^{z_0})^* \big)$ for all $z \in V$. But this proves the claim since $\Om^z (A^{z_0})^* = \Om^z \big((A^{z_0})^* - (B^{z_0})^* \big)$ for all $z \in V$. Indeed, this identity implies that $\T{dim}\big( \T{Im}(\Om^z (A^{z_0})^*) \big) \leq \T{dim}\big( (A^{z_0} - B^{z_0})^* \big)$ for all $z \in V$.
\end{proof}

We are now ready to prove the main result of this section:

\begin{prop}\label{p:perhol}
Suppose that $(d^z_j \op 0) - (\de^z_j \op 0) : X_j \op C_j \to X_{j-1} \op C_{j-1}$ has finite rank for all $z \in U$ and that the associated map $U \to \sL^1(X_j \op C_j , X_{j-1} \op C_{j-1})$ is holomorphic for all $j \in \zz$. Then the map $U \to \cc^*$ defined by $z \mapsto i_{\Ga}^{-1} P(\De^z \op \Ga, D^z \op C) i_C$ is holomorphic.
\end{prop}
\begin{proof}
Suppose first that $C = \Ga = \{0\}$.

Let $z_0 \in U$. It is enough to show that the map in question is holomorphic on an open neighborhood of $z_0 \in U$.

By Lemma \ref{l:holidetra} there exist an open neighborhood $V$ of $z_0 \in U$ and holomorphic idempotents $E_j\, , \, \Phi_j : V \to \sL(X_j)$ such that $\T{Im}(E_j^z) = \T{Ker}(d_j^z)$ and $\T{Im}(\Phi_j^z) = \T{Ker}(\de_j^z)$ for all $j \in \zz$ and all $z \in V$. Furthermore, we may suppose that $E_j -\Phi_j : V \to \sL^1(X_j)$ is holomorphic and factorizes through the finite rank operators $\sF(X_j)$ for all $j \in \zz$. Remark that we are relying on the assumption that there exists a $J \in \nn$ such that $X_j = \{0\}$ whenever $|j| \geq J$ at this point.

For each $j \in \zz$ and each $z \in V$, define the pseudo-inverses $(d_j^z)^\da : X_{j-1} \to X_j$ and $(\de_j^z)^\da : X_{j-1} \to X_j$ such that
\[
\begin{split}
& (d_j^z)^\da d_j^z = (1 - E_j^z) : X_j \to X_j \q d_j^z (d_j^z)^\da  = E_{j-1}^z : X_{j-1} \to X_{j-1} \q \T{and} \\
& (\de_j^z)^\da \de_j^z = (1 - \Phi_j^z) : X_j \to X_j \q \de_j^z (\de_j^z)^\da  = \Phi_{j-1}^z : X_{j-1} \to X_{j-1}
\end{split}
\]
It can then be verified that the associated maps $(d_j)^\da \, , \, (\de_j^\da) : V \to \sL(X_{j-1},X_j)$ are holomorphic.

Now, for each $j \in \zz$, we have that
\[
\begin{split}
d_j^\da - \de_j^\da & = d_j^\da ( \de_j - d_j ) \de_j^\da + d_j^\da (E_{j-1} - \Phi_{j-1})  +  (\Phi_j - E_j) \de_j^\da
\end{split}
\]
This implies that the difference $d_j^\da - \de_j^\da : V \to \sL^1(X_{j-1},X_j)$ is holomorphic and factorizes through the finite rank operators $\sF(X_{j-1},X_j)$ for all $j \in \zz$.

For each $z \in V$, let 
\[
\begin{split}
\si_+^z & := d_+^z + (d_-^z)^\da \, , \,  \tau_+^z := \de_+^z + (\de_-^z)^\da : X_+ \to X_- \q \T{and} \\
\si_-^z & := d_-^z + (d_+^z)^\da \, , \, \tau_-^z := \de_-^z + (\de_+^z)^\da : X_- \to X_+
\end{split}
\]
denote the associated isomorphisms of the $\zz/2\zz$-graded chains. The perturbation isomorphism is then given by
\[
P(\De^z,D^z) = \T{det}(\Si^z)^{-1} = \T{det}(\tau_-^z \si_+^z)^{-1} : V \to \cc^*
\]

It is now not hard to see, that the above considerations imply that the map $\Si - 1 : V \to \sL^1(X_+)$ is holomorphic. But this property guarantees that the Fredholm determinant yields a holomorphic map $\T{det}(\Si) : V \to \cc^*$, see \cite[Chapter IV, Section 1.8]{GoKr:ITN}. This proves the proposition in the case where $C = \Ga = \{0\}$.

To prove the general case, choose isomorphisms $F_+ : C_+ \to C_-$ and $G_+ : \Ga_+ \to \Ga_-$. This is possible since $\T{Ind}(C) = \T{Ind}(\Ga) = 0$. Consider the perturbed $\zz/2\zz$-graded Fredholm complexes
\[
\begin{CD}
(D^z \op C)_{\sF} : X_+ \op C_+ @>{d^z_+ \op F_+}>> X_- \op C_- @>{d^z_- \op 0}>> X_+ \op C_+ & \q \T{and} \\
(\De^z \op C)_{\sG} : Y_+ \op \Ga_+ @>{\de^z_+ \op G_+}>> Y_- \op \Ga_- @>{\de^z_- \op 0}>> X_- \op C_-
\end{CD}
\]
It then follows from Proposition \ref{p:pereaz} that
\[
P\big((D^z \op C)_{\sF},(D^z \op C)\big) i_C = \T{det}(F_+) \, \, \, \T{and} \,\, \,
P\big((\De^z \op \Ga)_{\sG}, (\De^z \op \Ga) \big) i_{\Ga} = \T{det}(G_+)
\]
for all $z \in U$. The transitivity of the perturbation isomorphism (Theorem \ref{t:algprocom}) now yields that
\[
\begin{split}
i_\Ga^{-1} P(\De^z \op \Ga, D^z \op C) i_C & = i_{\Ga}^{-1} P\big(\De^z \op \Ga, (\De^z \op \Ga)_{\sG} \big)
P\big((\De^z \op \Ga)_{\sG},(D^z \op C)_{\sF} \big) \\ 
& \qqq \ci P\big( (D^z \op C)_{\sF},D^z \op C\big) i_C \\
& = \T{det}(G_+)^{-1} P\big((\De^z \op \Ga)_{\sG},(D^z \op C)_{\sF} \big) \T{det}(F_+)
\end{split}
\]
for all $z \in U$. But this assignment depends analytically on $z \in U$ by the first part of the present proof.

\end{proof}

\section{Local trivializations of determinant line bundles}\label{s:loctridet}
\emph{Throughout this section $D := \{ D^z \}_{z \in U}$ will be a holomorphic family of \emph{Fredholm complexes}.}

In this section, we shall see how the concept of perturbation isomorphisms allows us to construct local trivializations of the holomorphic determinant line bundle associated to $D$. This very explicit form of the local trivializations will turn out to be an advantage for our investigation of the analyticity of the torsion isomorphisms.

For Fredholm operators the determinant line bundle was constructed by Quillen in \cite{Qui:DCR}. The construction was then generalized by Freed to the case of Fredholm complexes, see \cite[\S 2]{Fre:DLB}.

Let us start by recalling the algebraic structure of the determinant line bundle: It is given by the vector space $|D| := \coprod_{z \in U} |H_+(D^z)| \ot |H_-(D^z)|^*$ and the surjective map $|D| \to U$, $s_+^z \ot (s_-^z)^* \mapsto z$. 

In order to construct our local trivializations of the data $|D| \to U$ we need the following concept:

\begin{dfn}\label{d:loctrifre}
Let $V \su U$ be an open set. A \emph{local trivialization} of the holomorphic family of Fredholm complexes $D := \{ D^z \}$ over $V$ consists of
\begin{enumerate}
\item A trivial Fredholm complex 
\[
\begin{CD}
C : \q \ldots @<<{0}< \cc^{n_{j-1}} @<<{0}< \cc^{n_j} @<<{0}< \cc^{n_{j+1}} @<<{0}< \ldots   
\end{CD}
\]
\item A holomorphic map $F_j = \ma{cc}{F_j^1 \\ F_j^2} : V \to \sL\big(\cc^{n_j},X_{j-1} \op \cc^{n_{j-1}}\big)$ for each $j \in \zz$.
\end{enumerate}
such that the perturbed sequence
\[
\begin{CD}
D^z_\sF : \q \ldots @<<< X_{j-1} \op \cc^{n_{j-1}} @<<{\ma{cc}{d_j^z & (F_j^1)^z \\ 0 & (F_j^2)^z}}< X_j \op \cc^{n_j}
@<<< \ldots
\end{CD}
\]
is an exact chain complex for all $z \in V$. A local trivialization over $V \su U$ will be denoted by $\sF := \{F^z\}_{z \in V}$.
\end{dfn}

Let us immediately prove the existence of local trivializations of $D$ near any point $z_0 \in U$.

\begin{lemma}\label{l:loctrifre}
Let $z_0 \in U$. Then there exist an open neighborhood $V \su U$ of $z_0$ and a local trivialization of $D$ over $V$.
\end{lemma}
\begin{proof}
Without loss of generality we may assume that the chains $X_j$ are trivial for all $j < 0$. 

We first prove by induction on $k \in \nn_0$ that there exist
\begin{enumerate}
\item A trivial Fredholm complex
\[
\begin{CD}
C_k : \q \{0\} @<<< \cc^{n_1} @<<{0}< \ldots @<<{0}< \cc^{n_{k+1}} @<<< \{0\}
\end{CD}
\]
\item A holomorphic map $F_j = \ma{cc}{F_j^1 \\ F_j^2} : V \to \sL\big(\cc^{n_j},X_{j-1} \op \cc^{n_{j-1}}\big)$ for all $j \in \{1,\ldots,k + 1\}$, where $V \su U$ is an open neighborhood of $z_0$.
\end{enumerate}
such that the perturbed sequence
\[
\begin{CD}
D^z_{\sF_k} : \ldots @<<< X_{j-1} \op \cc^{n_{j-1}} @<<{\ma{cc}{d_j^z & (F_j^1)^z \\ 0 & (F_j^2)^z}}< 
X_j \op \cc^{n_j} @<<< \ldots
\end{CD}
\]
is a chain complex with $H_j( D^z_{\sF_k} ) = \{0\}$ for all $z \in V$ and all $j \in \{0,\ldots,k\}$.

Let $k \in \nn_0$ and suppose that $H_j(D^z) = \{0\}$ for all $z \in U$ and all $j \in \{-1,0,\ldots,k-1\}$. By Lemma \ref{l:holidetra} we may choose an open neighborhood $W \su U$ of $z_0$ and a holomorphic idempotent $E_k : W \to \sL(X_k)$ such that $\T{Im}(E^z_k) = \T{Ker}(d_k^z)$ for all $z \in W$. Remark here that the images of all the differentials $d_l^z$ are closed as a consequence of the Fredholmness assumption on $D$, see \cite[Theorem 2]{Cur:FOD}.

Choose an $n_{k+1} \in \nn_0$ and a linear map $F_{k+1} : \cc^{n_{k+1}} \to \T{Ker}(d_k^{z_0})$ such that $\ma{cc}{d_{k+1}^{z_0} & F_{k+1}} : X_{k+1} \op \cc^{n_{k+1}} \to \T{Ker}(d_k^{z_0})$ is surjective. It then follows that $\ma{cc}{d_{k+1}^z & E_k^z F_{k+1}} : X_{k+1} \op \cc^{n_{k+1}} \to \T{Ker}(d_k^z)$ is surjective for all $z$ in an open neighborhood $V \su W$ of $z_0$. This proves the above claim.

Without loss of generality, we may now suppose that there exists a $k \in \nn_0$ such that $X_j = \{0\}$ for all $j > k$ and $H_j(D^z) = \{0\}$ for all $j < k$ and all $z \in U$. As above, let $W \su U$ be an open neighborhood of $z_0$ and let $E_k : W \to \sL(X_k)$ be a holomorphic idempotent with $\T{Im}(E_k^z) = \T{Ker}(d_k^z) = H_k(D^z)$ for all $z \in W$. We may then choose an $n_{k + 1} \in \nn_0$ and an isomorphism $F_{k + 1} : \cc^{n_{k + 1}} \to \T{Ker}(d_k^{z_0})$. It follows that $E_k^z F_{k + 1} : \cc^{n_{k+1}} \to \T{Ker}(d_k^z)$ is an isomorphism for all $z$ in an open neighborhood $V \su W$ of $z_0$. This ends the proof of the lemma.
\end{proof}

We shall now see how a local trivialization $\sF$ of $D$ over an open set $V \su U$ gives rise to a local trivialization
\[
\Phi_{\sF} : \coprod_{z \in V} |H_+(D^z)| \ot |H_-(D^z)|^* \to V \ti \cc
\]
of the collection of determinant lines over $V$.

Let $\big( \cc^{n_+} \op \cc^{n_-} , 0 )$ denote the trivial $\zz/2\zz$-graded complex associated to $\sF$. The standard bases in $\cc^{n_+}$ and $\cc^{n_-}$ then provide us with non-trivial elements $\om_+ \in |\cc^{n_+}|$ and $\om_- \in |\cc^{n_-}|$. In particular we have a linear map
\[
i_C : |D^z| \to |D^z \op C| \q 
i_C : s_+^z \ot (s_-^z)^* \mapsto (s_+ \we \om_+) \ot (s_- \we \om_-)^*
\]
for all $z \in V$. Furthermore, for each $z \in V$, we have the perturbation isomorphism $P\big(D_{\sF}^z, D^z \op C \big) : |D^z \op C| \to \cc$. The local trivialization $\Phi_{\sF}$ is then defined by
\[
\Phi_{\sF} : s_+^z \ot (s_-^z)^* \mapsto \Big( z, \big( P(D_{\sF}^z, D^z \op C ) i_C \big)\big(s_+^z \ot (s_-^z)^*\big) \Big)
\]
for all $z \in V$.

We end this section by showing that our local trivializations define a holomorphic line bundle structure on the determinant lines, $|D| \to U$. This is the content of the next proposition:

\begin{prop}\label{p:loctridet}
Let $\sF_1$ and $\sF_2$ be two local trivializations of $D$ over the open subsets $V_1$ and $V_2$ of $U$. Then the composition $\Phi_{\sF_2} \Phi_{\sF_1}^{-1}$ defines a holomorphic map $V_1 \cap V_2 \to \cc^*$.
\end{prop}
\begin{proof}
Let $C_1$ and $C_2$ denote the trivial Fredholm complexes associated with $\sF_1$ and $\sF_2$. Choose trivial Fredholm complexes $C_3$ and $C_4$ of index zero such that $C_1 \op C_3 = C_2 \op C_4$.

By Lemma \ref{l:dirsumper} we have that
\[
\begin{split}
(\Phi_{\sF_2} \Phi_{\sF_1}^{-1})(z) 
& = P(D^z_{\sF_2}, D^z \op C_2) i_{C_2} \ci i_{C_1}^{-1} P(D^z \op C_1, D^z_{\sF_1}) \\
& = i_{C_4}^{-1} P(D^z_{\sF_2} \op C_4, D^z \op C_2 \op C_4) i_{C_4} i_{C_2} \\ 
& \qq \ci i_{C_1}^{-1} i_{C_3}^{-1} P(D^z \op C_1 \op C_3,D^z_{\sF_1} \op C_3) i_{C_3} \\
& = i_{C_4}^{-1} P(D^z_{\sF_2} \op C_4, D^z \op C_2 \op C_4) P(D^z \op C_1 \op C_3,D^z_{\sF_1} \op C_3) i_{C_3}
\end{split}
\]
for all $z \in V_1 \cap V_2$.

Thus, by the transitivity of the perturbation isomorphisms (Theorem \ref{t:algprocom}) it is enough to show that the map $V_1 \cap V_2 \to \cc^*$,  $z \mapsto i_{C_4}^{-1} P(D^z_{\sF_2} \op C_4, D^z_{\sF_1} \op C_3) i_{C_3}$ is holomorphic. But this is a consequence of Proposition \ref{p:perhol}.
\end{proof}

\section{Analyticity of the perturbation isomorphisms}\label{s:anaperiso}
\emph{Throughout this section, $D := \{D^z\} := \{X,d^z\}$ and $\De := \{\De^z\} := \{X,\de^z\}$ will be holomorphic families of Fredholm complexes parametrized by an open subset $U \su \cc^k$.} 

Let $|D|$ and $|\De|$ denote the holomorphic determinant line bundles associated with $D$ and $\De$. As a consequence of the investigations carried out in the preceding section we obtain the following:

\begin{thm}
Suppose that $\De$ is a finite rank perturbation of $D$ and that the associated maps $d_j - \de_j : U \to \sL^1(X_j,X_{j-1})$, $j \in \zz$, are holomorphic. Then the perturbation isomorphisms define a holomorphic map
\[
P(\De,D) : |D| \to |\De| \q P(\De,D) : s_+^z \ot (s_-^z)^* \mapsto P(\De^z , D^z)\big( s_+^z \ot (s_-^z)^* \big)
\]
\end{thm}
\begin{proof}
Let $z_0 \in U$ and let $\sF$ and $\sG$ be local trivializations of $D$ and $\De$ over an open neighborhood $V \su U$ of $z_0$. Let $C$ and $\Ga$ denote the associated trivial Fredholm complexes. As in the proof of Proposition \ref{p:loctridet}, we compute that
\[
\begin{split}
(\Phi_{\sG} P(\De,D) \Phi_{\sF}^{-1})(z)
& = i_{\Ga_1}^{-1} P(\De^z_{\sG} \op \Ga_1, \De^z \op \Ga \op \Ga_1)
P(\De^z \op \Ga \op \Ga_1, D^z \op C \op C_1) \\ 
& \qq \ci P(D^z \op C \op C_1, D^z_{\sF} \op C_1) i_{C_1} \\
& = i_{\Ga_1}^{-1} P(\De^z_{\sG} \op \Ga_1, D^z_{\sF} \op C_1) i_{C_1}
\end{split}
\]
for all $z \in V$, where $C_1$ and $\Ga_1$ are trivial Fredholm complexes of index zero with $\Ga \op \Ga_1 = C \op C_1$. The result of the theorem is now a consequence of Proposition \ref{p:perhol}.
\end{proof}

We remark that the above theorem is related to \cite{CaPi:PV}[Theorem 15]. Indeed, Carey and Pincus use perturbation isomorphisms of Fredholm operators to construct a global section of a certain pull back of Quillen's determinant line bundle.

\section{Analyticity of the torsion isomorphisms}\label{s:anatoriso}
\emph{Throughout this section, $D := \{D^z\} := \{X,d^z\}$ and $\De := \{\De^z\} := \{Y,\de^z\}$ will be holomorphic families of Fredholm complexes parametrized by an open subset $U \su \cc^k$. Furthermore, $A : D \to \De$ will be a holomorphic chain map.}

For each $z \in U$, we let $T(\G D^z) : |\De^z| \to |D^z| \ot |C^{A^z}|$ denote the torsion isomorphism of the mapping cone triangle
\[
\begin{CD}
\G D^z : D^z @>{A^z}>> \De^z @>{i}>> C^{A^z} @>{p}>> TD^z
\end{CD}
\]

\emph{The main result of this section is that these torsion isomorphisms induce a holomorphic map}
\[
T(\G D) : |\De| \to |D| \ot |C^A| \q T(\G D): (s_+^z \ot s_-^z) \to T(\G D^z)(s_+^z \ot s_-^z) 
\]
\emph{between the associated determinant line bundles.}

This proof of this result will rely on two of the previous main achievements of this paper. The first one is the description of the holomorphic structure of the involved determinant line bundles by means of perturbation isomorphisms, see Section \ref{s:loctridet}. The second one is Theorem \ref{t:pertor}, which provides a fundamental relationship between torsion isomorphisms and perturbation isomorphisms.

We will start by providing simultaneous local trivializations of $D$, $\De$ and $C^A$.

\begin{lemma}\label{l:trimapfre}
Let $z_0 \in U$. Then there exist local trivializations $(\sF,C)$ and $(\sG,\Ga)$ of $D$ and $\De$ over an open neighborhood $V \su U$ of $z_0$ together with holomorphic maps $H_j : V \to \sL(C_j, Y_j \op \Ga_j)$ such that $\Ga \op TC$ and the holomorphic maps
\begin{equation}\label{eq:loctrimap}
K_j = \ma{cc}{G_j^1 & H_{j-1}^1 \\ 0 & -F_{j-1}^1 \\ G_j^2 & H_{j-1}^2 \\ 0 & -F_{j-1}^2} : V \to \sL(\Ga_j \op C_{j-1}, Y_{j-1} \op X_{j-2} \op \Ga_{j-1} \op C_{j-2}) \q
\end{equation}
provide a local trivialization $(\sK,\Ga \op TC)$ of $C^A$ over $V$.
\end{lemma}
\begin{proof}
By Lemma \ref{l:loctrifre} there exists a local trivialization $(\sG,\Ga)$ of the holomorphic family of Fredholm complexes $\De$ over an open neighborhood $W \su U$ of $V$.

It then follows by construction that
\[
\begin{CD}
\ldots @<<< Y_{j-1} \op X_{j-2} \op \Ga_{j-1} @<{\ma{ccc}{\de_j & A_{j-1} & G_j^1 \\ 0 & -d_{j-1} & 0 \\ 0 & 0 & G_j^2}}<< Y_j \op X_{j-1} \op \Ga_j @<<< \ldots
\end{CD}
\]
defines a holomorphic family of Fredholm complexes over $W$. By another application of Lemma \ref{l:loctrifre} there exists a local trivialization $(\sL,TC)$ of this data over an open neighborhood $V \su W$ of $z_0$.

Let us write the associated holomorphic maps as $L_j = \ma{c}{H_{j-1}^1 \\ - F_{j-1}^1 \\ H_{j-1}^2 \\ - F_{j-1}^2} : V \to \sL(C_{j-1}, Y_{j-1} \op X_{j-2} \op \Ga_{j-1} \op C_{j-2})$. We then have that $(\sK,\Ga \op TC)$ is a local trivialization of the mapping cone complex $C^A$ over $V$, where the holomorphic maps $K_j : V \to \sL\big((\Ga \op TC)_j, (C^A)_{j-1} \op (\Ga \op TC)_{j-1} \big)$ are defined as in \eqref{eq:loctrimap}.

Notice now that it follows from the identity
\[
\ma{cccc}{\de_j & A_{j - 1} & G_j^1 & H_{j-1}^1 \\ 0 & -d_{j-1} & 0 & -F_{j-1}^1 \\ 0 & 0 & G_j^2 & H_{j-1}^2 \\ 0 & 0 & 0 & - F_{j-1}^2} 
\ma{cccc}{\de_{j+1} & A_j & G_{j+1}^1 & H_j^1 \\ 0 & -d_j & 0 & -F_j^1 \\ 0 & 0 & G_{j+1}^2 & H_j^2 \\ 0 & 0 & 0 & - F_j^2}
= 0
\]
that we have a holomorphic family of Fredholm complexes,
\[
\begin{CD}
D_{\sF} : \ldots @<<< X_{j-1} \op C_{j-1} @<{\ma{cc}{d_j & F_j^1 \\ 0 & F_j^2}}<< X_j \op C_j @<<< \ldots
\end{CD}
\]
and a holomorphic chain map $A_{\sH} : D_{\sF} \to \De_{\sG}$ defined by $(A_{\sH})_j := \ma{cc}{A_j & H_j^1 \\ 0 & H_j^2} : X_j \op C_j \to Y_j \op \Ga_j$. 

To prove the lemma, it therefore suffices to show that the homology of $D_{\sF}^z$ is trivial for all $z \in V$. But this follows by remarking that $C^A_{\sK}$ is chain isomorphic to the mapping cone of $A_{\sH} : D_{\sF} \to \De_{\sG}$. Indeed, this observation implies the existence of a six term exact sequence
\[
\begin{CD}
H_+(D_{\sF}^z) @>>> H_+(\De_{\sG}^z) @>>> H_+\big( (C^A_{\sK})^z\big) \\
@AAA & & @VVV \\
H_-\big((C^A_{\sK})^z \big) @<<< H_-(\De_{\sG}^z) @<<< H_-(D_{\sF}^z)
\end{CD}
\]
for all $z \in V$. This shows that $H_+(D_{\sF}^z) = H_-(D_{\sF}^z) = \{0\}$ for all $z \in V$, since these identities hold for both $\De_{\sG}^z$ and $(C^A_{\sK})^z$.
\end{proof}

We are now ready to state and prove the main result of this section.

\begin{thm}\label{t:torholo}
The map $T(\G D) : |\De| \to |D| \ot |C^A|$ given by $T(\G D) : s_+^z \ot (s_-^z)^* \mapsto T(\G D^z)\big(s_+^z \ot (s_-^z)^*\big)$ is holomorphic.
\end{thm}
\begin{proof}
Let $z_0 \in U$ and consider the local trivializations $(C,\sF)$, $(\Ga,\sG)$ and $(\Ga \op TC, \sK)$ of $D$, $\De$ and $C^A$ over an open neighborhood $V \su U$ of $z_0$, as constructed in Lemma \ref{l:trimapfre}.

We need to show that the composition $(\Phi_{\sF} \ot \Phi_{\sK})T(\G D) \Phi_{\sG}^{-1} : V \to \cc^*$ is holomorphic. We claim that it is constant and equal to the sign $(-1)^{n_- + m_+ (n_+ + n_-)}$, where $n_+ = \T{dim}(C_+)$, $n_- = \T{dim}(C_-)$ and $m_+ = \T{dim}(\Ga_+)$, $m_- = \T{dim}(\Ga_-)$.

To prove this claim, let $T(\wit{\G D}) : |\De \op \Ga| \to |D \op C| \ot |C^A \op \Ga \op TC|$ and $T(A \op 0) : |\De \op \Ga| \to |D \op C| \ot |C^{A \op 0}|$ denote the torsion isomorphisms associated with the triangles
\[
\begin{CD}
D \op C @>{(A \op 0)}>> \De \op \Ga @>>> C^A \op \Ga \op TC @>>> T(D \op C)
& & \q \T{and} \\
D \op C @>{(A \op 0)}>> \De \op \Ga @>>> C^{A \op 0} @>>> T(D \op C)
\end{CD}
\]
Furthermore, let $A_{\sH} : D_{\sF} \to D_{\sG}$ denote the holomorphic chain map constructed in the proof of Lemma \ref{l:trimapfre}. Finally, let $\Psi : C^A \op \Ga \op TC \to C^{A \op 0}$ and $\Psi : C^A_{\sK} \to C^{A_{\sH}}$ denote the holomorphic chain isomorphisms defined by interchanging the factors,
\[
\Psi_j : Y_j \op X_{j-1} \op \Ga_j \op C_{j-1} \to Y_j \op \Ga_j \op X_{j-1} \op C_{j-1} \q , \, \, j \in \zz
\]

Our claim is now proved by the following computation,
\[
\begin{split}
& (\Phi_{\sF} \ot \Phi_{\sK})T(\G D) \Phi_{\sG}^{-1} = 
\big(P(D_{\sF},D \op C) \ot P(C^A_{\sK},C^A \op \Ga \op TC) \big) \\
& \qqqq \qq \ci (i_C \ot i_{\Ga \op TC}) T(\G D) i_{\Ga}^{-1} P(\De \op \Ga,\De_{\sG}) \\
& \q = (-1)^{n_- + m_+ \cd (n_+ + n_-)}
\cd \big(P(D_{\sF},D \op C) \ot P(C^A_{\sK},C^A \op \Ga \op TC) \big) \\
& \qqq \ci T(\wit{\G D}) P(\De \op \Ga,\De_{\sG}) \\
& \q = (-1)^{n_- + m_+ \cd (n_+ + n_-)}
\cd \big(P(D_{\sF},D \op C) \ot P(C^{A_{\sH}},C^{A \op 0} \big) (1 \ot |\Psi|) \\
& \qqq \ci T(\wit{\G D}) P(\De \op \Ga,\De_{\sG}) \\
& \q = (-1)^{n_- + m_+ \cd (n_+ + n_-)}
\cd \big(P(D_{\sF},D \op C) \ot P(C^{A_{\sH}},C^{A \op 0}) \big) \\
& \qqq \ci T(A \op 0) P(\De \op \Ga,\De_{\sG}) \\
& \q = (-1)^{n_- + m_+ \cd (n_+ + n_-)}
\end{split}
\]
where the second identity follows from Lemma \ref{l:dilzer}, and the last identity follows from Theorem \ref{t:pertor} . The remaining identities can be verified by straightforward investigations. This proves the theorem.
\end{proof}


\def\cprime{$'$} \def\cprime{$'$}
\providecommand{\bysame}{\leavevmode\hbox to3em{\hrulefill}\thinspace}
\providecommand{\MR}{\relax\ifhmode\unskip\space\fi MR }
\providecommand{\MRhref}[2]{%
  \href{http://www.ams.org/mathscinet-getitem?mr=#1}{#2}
}
\providecommand{\href}[2]{#2}

\bibliographystyle{amsalpha-lmp}

\begin{thebibliography}{\textsc{KnMu76}}

\bibitem[\textsc{Bre11}]{Bre:DTC}
\textsc{M.~Breuning}, \emph{Determinant functors on triangulated categories},
  J. K-Theory \textbf{8} (2011), no.~2, 251--291. \MR{2842932}

\bibitem[\textsc{Bro75}]{Br:OAK}
\textsc{L.~G. Brown}, \emph{Operator algebras and algebraic {$K$}-theory},
  Bull. Amer. Math. Soc. \textbf{81} (1975), no.~6, 1119--1121. \MR{0383090 (52
  \#3971)}

\bibitem[\textsc{CaPi99a}]{CaPi:JTS}
\textsc{R.~Carey} and \textsc{J.~Pincus}, \emph{Joint torsion of {T}oeplitz
  operators with {$H\sp \infty$} symbols}, Integral Equations Operator Theory
  \textbf{33} (1999), no.~3, 273--304. \MR{1671481 (2000f:47050)}

\bibitem[\textsc{CaPi99b}]{CaPi:PV}
\bysame, \emph{Perturbation vectors}, Integral Equations Operator Theory
  \textbf{35} (1999), no.~3, 271--365. \MR{1716541 (2001e:47048)}

\bibitem[\textsc{CaPi06}]{CaPi:STT}
\textsc{R.~W. Carey} and \textsc{J.~D. Pincus}, \emph{Steinberg symbols modulo
  the trace class, holonomy, and limit theorems for {T}oeplitz determinants},
  Trans. Amer. Math. Soc. \textbf{358} (2006), no.~2, 509--551 (electronic).
  \MR{2177029 (2006j:47044)}

\bibitem[\textsc{CoKa88}]{CoKa:CMF}
\textsc{A.~Connes} and \textsc{M.~Karoubi}, \emph{Caract\`ere multiplicatif
  d'un module de {F}redholm}, $K$-Theory \textbf{2} (1988), no.~3, 431--463.
  \MR{972606 (90c:58174)}

\bibitem[\textsc{Cur81}]{Cur:FOD}
\textsc{R.~E. Curto}, \emph{Fredholm and invertible {$n$}-tuples of operators.
  {T}he deformation problem}, Trans. Amer. Math. Soc. \textbf{266} (1981),
  no.~1, 129--159. \MR{613789 (82g:47010)}

\bibitem[\textsc{Fre87}]{Fre:DLB}
\textsc{D.~S. Freed}, \emph{On determinant line bundles}, Mathematical aspects
  of string theory ({S}an {D}iego, {C}alif., 1986), Adv. Ser. Math. Phys.,
  vol.~1, World Sci. Publishing, Singapore, 1987, pp.~189--238. \MR{915823}

\bibitem[\textsc{GoKr69}]{GoKr:ITN}
\textsc{I.~C. Gohberg} and \textsc{M.~G. Kre{\u\i}n}, \emph{Introduction to the
  theory of linear nonselfadjoint operators}, Translated from the Russian by A.
  Feinstein. Translations of Mathematical Monographs, Vol. 18, American
  Mathematical Society, Providence, R.I., 1969. \MR{0246142 (39 \#7447)}

\bibitem[\textsc{Kaa12}]{Kaa:JTS}
\textsc{J.~Kaad}, \emph{Joint torsion of several commuting operators}, Adv.
  Math. \textbf{229} (2012), no.~1, 442--486. \MR{2854180}

\bibitem[\textsc{KnMu76}]{KnMu:PMS}
\textsc{F.~F. Knudsen} and \textsc{D.~Mumford}, \emph{The projectivity of the
  moduli space of stable curves. {I}. {P}reliminaries on ``det'' and
  ``{D}iv''}, Math. Scand. \textbf{39} (1976), no.~1, 19--55. \MR{0437541 (55
  \#10465)}

\bibitem[\textsc{Kvi85}]{Qui:DCR}
\textsc{D.~Kvillen}, \emph{Determinants of {C}auchy-{R}iemann operators on
  {R}iemann surfaces}, Funktsional. Anal. i Prilozhen. \textbf{19} (1985),
  no.~1, 37--41, 96. \MR{783704 (86g:32035)}

\bibitem[\textsc{Qui73}]{Qui:HAK}
\textsc{D.~Quillen}, \emph{Higher algebraic {$K$}-theory. {I}}, Algebraic
  {$K$}-theory, {I}: {H}igher {$K$}-theories ({P}roc. {C}onf., {B}attelle
  {M}emorial {I}nst., {S}eattle, {W}ash., 1972), Springer, Berlin, 1973,
  pp.~85--147. Lecture Notes in Math., Vol. 341. \MR{0338129 (49 \#2895)}

\bibitem[\textsc{Wid90}]{Wid:TAT}
\textsc{H.~Widom}, \emph{Eigenvalue distribution of nonselfadjoint {T}oeplitz
  matrices and the asymptotics of {T}oeplitz determinants in the case of
  nonvanishing index}, Topics in operator theory: {E}rnst {D}. {H}ellinger
  memorial volume, Oper. Theory Adv. Appl., vol.~48, Birkh\"auser, Basel, 1990,
  pp.~387--421. \MR{1207410 (93m:47033)}

\end{thebibliography}

\end{document}